\documentclass[11pt]{amsart}
\usepackage{amsmath, amsthm, amssymb, mathtools}
\usepackage[utf8]{inputenc}
\usepackage{hyperref}
\usepackage{enumitem}
\usepackage{mathrsfs}
\usepackage{fullpage}
\usepackage{tikz-cd} 
\usepackage{verbatim}

\usepackage{bbm}
\usepackage{MnSymbol}

\theoremstyle{plain}
\newtheorem{theorem}{Theorem}[section]
\newtheorem{proposition}[theorem]{Proposition}
\newtheorem*{proposition*}{Proposition}

\usepackage[backend=biber, style=alphabetic,sorting=nyt,maxnames=99,maxalphanames=99]{biblatex}
\newtheorem{lemma}[theorem]{Lemma}
\newtheorem{corollary}[theorem]{Corollary}
\newtheorem{question}[theorem]{Question}
\newtheorem{conjecture}[theorem]{Conjecture}
\theoremstyle{definition}
\newtheorem{definition}[theorem]{Definition}
\newtheorem{goal}[theorem]{Goal}
\newtheorem{example}[theorem]{Example}
\newtheorem{remark}[theorem]{Remark}

\begin{document}

\title{On the tangent bundle and the divisor theory of a general matroid}
\author{Ronnie Cheng}
\address{Department of Mathematics, Stanford University}
\email{rtcheng@stanford.edu}
\date{Oct 7, 2025}

\begin{abstract}
Extending classical algebro-geometric constructions to arbitrary matroids, we construct a $K$-class $T_M\in K(M)$ for every loopless matroid $M$. When $M$ is realizable by a linear subspace $L$, $T_M$ recovers the $K$-class of the tangent bundle of the wonderful compactification $W_L$. We derive two formulas for the total Chern class of $T_M$ (one combinatorial and one geometric) and show that the associated Todd class agrees with the Todd class appearing in the matroid Hirzebruch--Riemann--Roch formula. To develop a positivity theory entirely at the combinatorial level, we introduce the notion of a ``fake effective cone,'' a combinatorial analogue of the classical effective cone, and use it to characterize big and nef divisors in $A(M)$. Finally, we define the $\beta_S$ classes, obtained from Cremona conjugates of the classical $\alpha_S$ classes, and study their properties to provide a rich and computable family of combinatorially nef divisors.
\end{abstract}
\maketitle

\tableofcontents
\section{Introduction}

A (loopless) matroid $M$ on a finite ground set $E$ carries a rich geometric and combinatorial
structure. Associated with $M$, one has the Bergman fan $\Sigma_M$ and its associated toric variety, which we denote for simplicity by $X_M$ (rather than $X_{\Sigma_M}$). We work over an arbitrary field $\mathbbm{k}$. When $M$ is realizable by a linear
subspace $L\subseteq \mathbbm{k}^E$, the De Concini--Procesi wonderful compactification $W_L$ \cite{DCP} provides a
smooth projective model that sits naturally inside the permutahedral variety $X_E$, and the inclusion factors through $X_M$:
\[
W_L \;\hookrightarrow\; X_M \;\hookrightarrow\; X_E.
\]

In the realizable case, the Chow and $K$-rings of $X_M$ agree with those of $W_L$; geometric tools (such as Poincar\'e duality, Hodge--Riemann relations, and Hirzebruch--Riemann--Roch) therefore can be transferred to the combinatorial setting.

Extending geometric features from the realizable case to general matroids is a central theme in recent work in matroid theory. Following the foundational Hodge-theoretic breakthroughs of Adiprasito, Huh, and Katz \cite{AHK18} and subsequent developments in matroid $K$-theory by Larson, Li, Payne, and Proudfoot \cite{LLPP2024}, much of the ``projective geometry'' of Chow rings admits a purely combinatorial incarnation.
Nevertheless, while specific components of this picture have been extensively studied (such as the Todd class of the permutahedral variety by Castillo and Liu \cite{Todd}), several natural geometric objects, most notably the tangent bundle and the associated Todd class for arbitrary (possibly nonrealizable) matroids, have not yet been given a fully satisfactory analogue. The principal aim of this paper is to fill that gap.

\medskip

Below, we summarize the main results. Precise statements and proofs appear in the indicated locations within the paper. As the matroid $K$-ring $K(X_M)$ admits standard $\lambda$-ring operations (analogous to the geometric setting), geometric constructions such as duals $(\cdot)^\vee$, exterior powers $\wedge^i$, and the rank function extend naturally to virtual classes such as our proposed tangent class.

\begin{theorem}[Tangent class, Todd class, and Chow polynomial]\label{thm:main-tangent}
For every loopless matroid $M$, there exists a canonical $K$-class
$\widetilde T_M\in K(X_E)$ whose restriction $T_M := i_M^*\widetilde T_M\in K(X_M)$
(Definition~\ref{tangent bundle def}) satisfies the following properties:
\begin{enumerate}
  \item[(1)] \emph{Realizable compatibility.} If $M$ is realizable by $L\subseteq \mathbbm{k}^E$, then $T_M$ coincides with the $K$-class of the tangent bundle of the wonderful compactification $W_L$.
  \item[(2)] \emph{Chern class.} We provide two formulas for the Chern class of $T_M$ (Theorem~\ref{Chris tangent formula} and Theorem~\ref{tangent bundle formula}).
  \item[(3)] \emph{Todd class and HRR.} The Todd class associated with $T_M$ agrees with the
    Todd class appearing in Proposition~\ref{prop: HRR} (the Hirzebruch–Riemann–Roch theorem in the matroid setting); see Corollary~\ref{Todd class formula}.
  \item[(4)] \emph{Chow polynomial via Euler characteristics.} The Euler characteristics of the exterior
    powers of the cotangent class $\Omega_M := T_M^\vee$ recover the coefficients of the Chow polynomial:
    for all $i$,
  \[
    \dim A^{i}(M) = (-1)^i \chi\big(X_M,\wedge^i \Omega_M\big)
    \;=\;
    (-1)^i\deg\!\big(\operatorname{ch}(\wedge^i \Omega_M)\cdot \operatorname{td}(T_M)\big).
  \]
  See Theorem~\ref{Chow polynomial formula}.
\end{enumerate}
\end{theorem}

\medskip

Understanding the Todd class in the matroid Hirzebruch--Riemann--Roch formula is useful for
computing Euler characteristics and for formulating vanishing statements in analogy with the
Kawamata--Viehweg vanishing theorem (motivated by Eur and Larson's positivity results in \cite{EL}). This motivates a study of big and nef divisors in the
matroid setting. After comparing several candidate notions of nefness (and showing they are not equivalent;
see Theorem~\ref{thm:nef-definitions}), we introduce an operational notion of \emph{combinatorially big and nef} divisors (Definition~\ref{def:big and nef}). Within this framework, we prove the following results.

\begin{theorem}[Big and nef classes and the fake effective cone]\label{thm:main-positivity}
\noindent
\begin{enumerate}
  \item[(1)] \emph{Realizable compatibility.} If $M$ is realizable by $L\subseteq \mathbbm{k}^E$
  and $D$ is combinatorially big and nef for $M$, then the corresponding divisor on the wonderful
  compactification $W_L$ is big and nef in the classical sense.
  \item[(2)] \emph{Intersection inequalities.} We prove a collection of intersection-theoretic
  inequalities for nef and big-and-nef divisors. 
  \item[(3)] \emph{Fake effective cone.} There exists a combinatorial \emph{fake effective cone}
  whose role is analogous to that of the classical effective cone in defining and detecting big and
  nef divisors (see Theorem~\ref{fake effective cone}).
  \item[(4)] \emph{Matroid KV-type vanishing (rank 3).} We formulate a matroid analogue of the
  Kawamata--Viehweg vanishing statement, propose a weakened version, and verify the rank 3 case.
\end{enumerate}
\end{theorem}

A recurring practical question in positivity theory is the following: There are many nef classes
in the matroid Chow ring, but which of these can we effectively compute and use to test conjectures?
The classical $\alpha_F$-classes furnish important and highly computable examples, but they form a
rather small, highly structured family and do not provide enough flexibility for exploration of the
broader nef cone.

To produce a richer supply of controlled nef classes, we study the $\beta$-classes. While defined independently, these can be immediately identified as the Cremona conjugates of the $\alpha$-classes on the permutahedral variety. The Cremona description
is convenient conceptually, but the crucial point is that the $\beta$-classes give many new examples
of divisors that are often nef and that can be computed in concrete cases. For a nonempty set $S\subseteq E$, we define the class $\beta_S$ (Definition~\ref{beta def}), which satisfies the following properties.

\begin{theorem}[Properties of the $\beta$-classes]\label{thm:prop of beta}
\noindent
\begin{enumerate}
  \item[(1)] Under the standard Cremona involution $\operatorname{crem}: X_E \rightarrow X_E$, we have $\beta_S = \operatorname{crem}^*(\alpha_S)$ on $X_E$.
  \item[(2)] Let $[\mathcal{L}_S], [\mathcal{K}_S] \in K(X_M)$ be the $K$-classes of the line bundles corresponding to $\alpha_S$ and $\beta_S$, respectively. The exceptional isomorphism $\zeta_M$ (defined in Section~\ref{subsec: K Chern})
    sends $[\mathcal{L}_S]^{-1}$ to $1-\alpha_S$ and sends $[\mathcal{K}_S]$ to $1+\beta_S$
    (see Theorem~\ref{exceptional morphism of beta}).
  \item[(3)] Let $M$ be a rank $r$ matroid and let $S_1,\dots,S_{r-1}\subseteq E$ be nonempty subsets.
    Then
    \[
      \deg(\beta_{S_1}\cdots\beta_{S_{r-1}})>0 \quad\Longleftrightarrow\quad
      \deg(\alpha_{S_1}\cdots\alpha_{S_{r-1}})>0.
    \]
  \item[(4)] A weakened Kawamata--Viehweg vanishing statement holds for positive integral linear
    combinations of the $\alpha_S$ and $\beta_S$ (see Corollary \ref{weak KV for beta}).
\end{enumerate}
\end{theorem}

\subsection*{Organization of the paper}
Section~\ref{sec:preliminaries} fixes notation and recalls necessary facts regarding the geometry and combinatorics of matroids. 
Section~\ref{sec:main result} defines the tangent and Todd classes, proving that the relevant functions are valuative in order to extend multiple results from realizable matroids to arbitrary ones. 
Section~\ref{sec:realizable} computes the total Chern class via blow-ups and provides formulas for $c(T_M)$. 
Section~\ref{sec:nef} develops the positivity theory, introduces combinatorially big and nef divisors alongside the fake effective cone, and discusses the Kawamata--Viehweg vanishing statement. 
Finally, Section~\ref{sec:Equations} introduces and studies the properties of the $\beta$-classes. 
Importantly, Sections~\ref{sec:nef} and~\ref{sec:Equations} are largely self-contained and can be read independently of Sections~\ref{sec:main result} and~\ref{sec:realizable}.


\section*{Acknowledgements}
The author is grateful to his advisor, Ravi Vakil, for continuous guidance and suggestions throughout this work. The author thanks Matt Larson for his helpful suggestions and insightful discussions, which greatly improved the presentation of this paper. He is also grateful to Christopher Eur for mentioning the formula for $c(Q_M)$ in Theorem~\ref{thm:stair}. Finally, the author extends his sincere gratitude to the anonymous referees for their careful reading of the manuscript and their valuable suggestions, which have significantly improved the quality and clarity of this paper.

\section{Preliminaries}\label{sec:preliminaries}
Throughout this article, we assume that the reader is familiar with the main terminology in matroid theory. See \cite{Ox} for a general background. All matroids in the paper are assumed to be loopless unless otherwise mentioned. In this section, we establish the foundational background and notation used throughout the paper. We recall the standard algebraic and geometric frameworks for matroids—namely the wonderful compactification, the Chow ring, and the $K$-ring. We also review key operational tools required for our computations, including the valuative property and various morphisms such as pullback, pushforward, and deletion maps.
\subsection{The wonderful variety}

Let \(E=\{1,\dots,n\}\) be the ground set, let \(\mathbbm{k}\) be a field, and let \(L\subseteq \mathbbm{k}^E\) be a linear subspace of dimension \(r\) that is not contained in any coordinate hyperplane. For \(S\subseteq E\), set \(L_S:=L\cap \mathbbm{k}^{E\setminus S}\), where \(\mathbbm{k}^{E\setminus S}\) denotes the
coordinate subspace spanned by the coordinates indexed by \(E\setminus S\). A subset \(F\subseteq E\) is a \emph{flat} if it is maximal among subsets yielding the same subspace \(L_F\). The \emph{wonderful compactification} \(W_L\) is obtained from \(\mathbb{P}L\) by iteratively blowing up the linear subspaces \(\mathbb{P}L_S\) corresponding to nontrivial flats, proceeding in increasing order of dimension and taking strict transforms at each step; see \cite{DCP}. Equivalently, one may describe the construction as: first blow up all \(\{\mathbb{P}L_S:\ \dim L_S=1\}\) (points), then blow up the strict transforms of the loci \(\{\mathbb{P}L_S:\ \dim L_S=2\}\) (lines), and so on. The result
is a smooth projective variety of dimension \(r-1\), which we denote \(W_L\). 

\subsection{Matroids and their geometric realizations}

For $S \subseteq E$, we define $e_S = \sum_{i \in S} e_i$. Let $M$ be a matroid on the ground set $E$. If $\mathcal{F} = (\varnothing \subsetneq F_1 \subsetneq \cdots \subsetneq F_k \subsetneq E)$ is a flag of (nonempty) proper flats of $M$, we let $\rho_\mathcal{F}$ be the cone in $\mathbb{R}^E/\mathbb{R}(1, \ldots, 1)$ generated by $\{e_F \mid F \in \mathcal{F}\}$.
\begin{definition}
The \emph{Bergman fan} of a matroid $M$, denoted $\Sigma_M$, is the fan in $\mathbb{R}^E/\mathbb{R}(1, \ldots, 1)$ whose cones are $\{\rho_\mathcal{F} \mid \mathcal{F} \text{ flag of flats of } M\}$.
\end{definition}

In this way, a matroid $M$ is associated with a toric variety $X_{\Sigma_M}$, which we will denote by $X_M$. The \emph{Boolean matroid} $U_n$ on the ground set $E = \{1, \ldots, n\} = [n]$ is the matroid in which every subset is a flat. The toric variety $X_E$ associated with $U_n$ is called the \emph{permutahedral variety}. For any matroid $M$ on the ground set $E$, $\Sigma_M$ is a subfan of $\Sigma_{U_n}$, and thus the toric variety $X_{M}$ is an open subvariety of $X_E$. In particular, it is smooth.

For $L = \mathbbm{k}^E$, the wonderful compactification $W_L$ equals the toric variety $X_E$, and for any $L \subseteq \mathbbm{k}^E$, the inclusion $L \hookrightarrow \mathbbm{k}^E$ induces an inclusion $W_L \hookrightarrow X_E$.  

We say that a matroid $M$ is \emph{realizable} or \emph{realized} by $L \subseteq \mathbbm{k}^E$ if the flats of $M$ come from the flats of $L$. We say such $L$ is a \emph{realization} of $M$. For a realizable matroid $M$, the realization $L$ is not unique, and the corresponding varieties $W_L$ may differ. Nonetheless, all such $W_L$ share the same cohomology ring.

\begin{proposition}[{\cite[Remark 2.13]{BHM}}]
Let $L \subseteq \mathbbm{k}^E$ be a realization of a matroid $M$. Then the
inclusion $W_L \hookrightarrow X_E$ factors through $X_{M}$, and the pullback map $$A^\ast(M) \xrightarrow{\sim} A^\ast(W_L)$$ is an isomorphism. Here, $A^\ast(M)$ denotes the Chow ring $A^\ast(X_{M})$.
\end{proposition}

This result extends to $K$-rings (the Grothendieck ring of vector bundles) as follows.

\begin{proposition}[{\cite[Proposition 1.6]{LLPP2024}}]
Let $L \subseteq \mathbbm{k}^E$ be a realization of a matroid $M$. Then the
restriction map $$K(M) \xrightarrow{\sim} K(W_L)$$ is an isomorphism. Here, $K(M)$ denotes the $K$-ring $K(X_{M})$.
\end{proposition}

\subsection{Chow ring of a matroid} \label{subsec:Chow ring of a matroid}

The Chow ring $A^*(M)$ admits a presentation as a quotient of a polynomial ring \cite{FY04, AHK18}:

\[
A^*(M) = \mathbb{Z}[x_F \mid F \text{ a nonempty proper flat of } M] / (I + J),
\]

where:
\begin{itemize}
    \item $I$ is the ideal generated by products $x_F x_G$ for incomparable flats $F$ and $G$,
    \item $J$ is the ideal generated by linear forms $\sum_{F \ni i} x_F - \sum_{F \ni j} x_F$ for each pair of elements $i, j \in E$.
\end{itemize}

Throughout, whenever a variable \(x_F\) appears, \(F\) is understood to be a nonempty proper flat unless otherwise stated. When $M$ is realized by $L$, the divisor $x_F \in A^*(W_L)$ is the class of the exceptional divisor when blowing up at $\mathbb{P}L_F$ and then taking the pullbacks under further blow-ups. In general, $X_{M}$ is not projective. In their landmark work, Adiprasito, Huh, and Katz \cite{AHK18} proved that the Chow ring of a matroid satisfies several properties. In particular, the Chow ring vanishes in degree $\geq r$, and there is a degree map $\deg_M \colon A^{r-1}(M) \to \mathbb{Z}$ that is an isomorphism, mapping any monomial corresponding to a complete flag of flats to 1. The Chow ring also satisfies Poincar\'e duality in the sense that $$A^i(M) \times A^{r-1-i}(M) \rightarrow A^{r-1}(M) \xrightarrow{\sim} \mathbb{Z}$$ is a perfect pairing for $0 \leq i \leq r-1$, where the last map is given by $\deg_M$.

\begin{definition} \label{alpha def}
We define $\alpha = \alpha_i =  \sum_{F \ni i} x_F \in A^1(M)$ and $\beta = \beta_i = \sum_{F \not\ni i} x_F \in A^1(M)$. Observe that the difference $\alpha_i - \alpha_j$ vanishes in the Chow ring, so $\alpha$ is independent of the element $i$ we choose. Similarly, $\beta$ is also independent of $i$.

For a set $\varnothing \subsetneq S \subseteq E$, we define

$$\alpha_S = \alpha - \sum_{\text{proper flats }F \supseteq S} x_F.$$

For example, $\alpha = \alpha_E$.
\end{definition}
The following definition is new.
\begin{definition} \label{beta def}
For a set $\varnothing \subsetneq S \subseteq E$, we define

$$\beta_S = \beta - \sum_{\text{proper flats }F \subseteq E \setminus S} x_F.$$

For example, $\beta = \beta_E$.
\end{definition}
Note that $\alpha_S = \alpha_{\operatorname{cl}(S)}$, where $\operatorname{cl}(S)$ (the closure of $S$) is the smallest flat that contains $S$. In contrast, $\beta_S$ cannot always be written as $\beta_F$ for a flat $F$. We also write $\alpha_M$ or $\beta_M$ for the $\alpha$ and $\beta$ classes for the matroid $M$ when there are multiple different matroids being discussed. We will try to make the context clear to avoid confusion between two potential meanings of the subscripts.

In the case of realizable matroids, $\alpha$ is the class of the vanishing locus of a generic section of the pullback of $\mathcal{O}(1)$ from the projective space $\mathbb{P}(\mathbbm{k}^E)$ along the iterated blow-ups. More precisely, $\alpha_F$ is the pullback of the hyperplane section that passes through the blow-up center corresponding to the flat $F$. (The class $\alpha_F$ is sometimes called $h_F$ in other papers.) 

For $E = \{1, \ldots, n\}$, we have a natural birational morphism $X_E \rightarrow \mathbb{P}\mathbbm{k}^E \cong \mathbb{P}^{n-1}$ given by the blow-up construction, where the blow-up centers are the intersections of coordinate hyperplanes $x_i = 0$. The Cremona involution $$[t_1:\cdots:t_{n}] \dashrightarrow [t_1^{-1}: \cdots : t_n^{-1}]$$ induces the Cremona involution $\operatorname{crem}: X_E \rightarrow X_E$ (see \cite[Section 2.6]{BEST}). The Cremona map exchanges $x_I$ and $x_{E \setminus I}$. Therefore, in $X_E$ we have $\beta_S = \operatorname{crem}^*(\alpha_S)$ for any set $\varnothing \subsetneq S \subseteq E$. In particular, $\beta = \operatorname{crem}^*(\alpha)$.

For any matroid $M$ on the ground set $E$ with the inclusion $i:X_{M} \rightarrow X_E$, the pullback map $i^\ast$ sends $x_F$ to $x_F$ if $F$ is a flat of $M$, and 0 otherwise. Therefore, for any set $\varnothing \subsetneq S \subseteq E$, $\alpha_S$ and $\beta_S$ in $M$ are the pullbacks of $\alpha_S$ and $\beta_S$ from $X_E$, respectively.

\subsection{K-theory and the Chern character} \label{subsec: K Chern}

Larson, Li, Payne, and Proudfoot introduced the $K$-ring of a general matroid in \cite{LLPP2024}. The Chern character map $\operatorname{ch} \colon K(M) \to A^\ast(M) \otimes \mathbb{Q}$, which sends the $K$-class of a line bundle $\mathcal{L}$ to $\exp(c_1(\mathcal{L}))$, induces a ring isomorphism over $\mathbb{Q}$ \cite[Example 15.2.16(b)]{Fulton}:
\[
\operatorname{ch}: K(M)\otimes \mathbb{Q} \xrightarrow{\sim} A^*(M) \otimes \mathbb{Q}.
\]

Let $\mathcal{L}_F$ denote the line bundle associated with the divisor $\alpha_F$. There is also a ring isomorphism \cite[Theorem 1.8]{LLPP2024} \[
\zeta_M: K(M) \xrightarrow{\sim} A^*(M) ,
\]
called the \emph{exceptional isomorphism} sending $$[\mathcal{L}_F] \mapsto \frac{1}{1-\alpha_F} = 1 + \alpha_F + \alpha_F^2 + \cdots.$$ It turns out that the classes $[\mathcal{L}_F]$ generate the $K$-ring of $M$, and $\zeta_M$ is determined by the image $\zeta_M([\mathcal{L}_F])$.

In \cite{LLPP2024}, the authors defined the Euler characteristic map $\chi:K(M) \rightarrow \mathbb{Z}$ by sending the class $K$ to 
\begin{equation} \label{Euler char exceptional}
    \chi(K) = \deg_M\left(\zeta_M(K)\cdot(1 + \alpha + \alpha^2 + \cdots)\right) \in \mathbb{Z},
\end{equation} 
which agrees with the Euler characteristic map when $M$ is realizable. Using Poincar\'e duality, one shows the following. 

\begin{proposition}[{Matroid Hirzebruch--Riemann--Roch, \cite[Proposition 4.11]{EL}}]\label{prop: HRR}
For a matroid $M$, there is a unique Todd class $\operatorname{Todd}_M \in A^\ast(M)_{\mathbb{Q}}$ such that for $K \in K(M)$, we have $$\chi(K) = \deg_M\left(\operatorname{ch}(K) \cdot \operatorname{Todd}_M\right).$$ 

Moreover, the degree 0 part of $\operatorname{Todd}_M$ is 1. 
\end{proposition}

The existence of the class $\operatorname{Todd}_M$ comes from Poincar\'e duality, and the degree 0 part of $\operatorname{Todd}_M$ can be deduced by considering $K = \mathcal{L}_E^{\otimes k}$. We will prove that it is indeed the Todd class of the tangent $K$-class we define in Section~\ref{subsec:HRR} (Corollary \ref{Todd class formula}). The Todd class for the special case where $M = U_{n}$ was extensively studied in \cite{Todd}.

\subsection{Maps between Chow rings of matroids}

Let $M = (E, \mathcal{I})$ be a matroid and let $S \subseteq E$. We denote by $M^S$, $M \setminus S$, and $M_S$ the restriction, deletion, and contraction, respectively.

\begin{lemma}[{Pullback map, \cite[Section 2.6]{BHM}}]\label{lem:pullback}
Let $M$ be a matroid of rank $r$, and let $F$ be a (nonempty) proper flat of $M$.  There is a unique graded algebra homomorphism
\[
\varphi^F_M \colon A^\ast(M) \;\longrightarrow\; A^\ast(M_F)\,\otimes\,A^\ast(M^F),
\]
called the \emph{pullback map}, such that for each (nonempty) proper flat \(G \neq F\) of \(M\),
\[
\varphi^F_M(x_G)
=
\begin{cases}
0, & F\text{ and }G\text{ are incomparable},\\
1\otimes x_G, & G\subsetneq F,\\
x_{G \setminus F}\otimes 1, & F\subsetneq G.
\end{cases}
\]
Moreover, \(\varphi^F_M\) is surjective and additionally satisfies
\begin{align*}
\varphi^F_M(x_F)
  &= -\bigl(1\otimes \alpha_{M^F} \;+\;\beta_{M_F}\otimes 1\bigr),\\
\varphi^F_M(\alpha_M)
  &= \alpha_{M_F}\otimes 1,\\
\varphi^F_M(\beta_M)
  &= 1\otimes \beta_{M^F}.
\end{align*}

\end{lemma}
There is also a pushforward map.

\begin{lemma}[{Pushforward map, \cite[Section 2.6]{BHM}}]\label{lem:pushforward}
Let $M$ be a matroid of rank $r$, and let $F$ be a (nonempty) proper flat of $M$.  There is a unique graded algebra homomorphism
\[
\psi^F_M \colon A^\ast(M_F)\,\otimes\,A^\ast(M^F) \;\longrightarrow\; A^\ast(M),
\]
called the \emph{pushforward map}, which satisfies, for any collection $\mathcal{S}_1$ of proper flats of $M$ strictly containing $F$ and any collection $\mathcal{S}_2$ of (nonempty) proper flats of $M$ strictly contained in $F$,
\[
\psi^F_M\big(\prod_{F'\in \mathcal{S}_1}x_{F'\setminus F} \otimes \prod_{F''\in\mathcal{S}_2} x_{F''}\big) = x_F \prod_{F'\in \mathcal{S}_1}x_{F'}\prod_{F''\in\mathcal{S}_2} x_{F''}.
\]

The composition $\psi^F_M \circ \varphi^F_M$ is multiplication by the element $x_F$, and the composition $\varphi^F_M \circ \psi^F_M$ is multiplication by the element $\varphi^F_M(x_F)$. In particular, for $f \in A^{r-2}(M)$, we have $\deg_M(f \cdot x_F) = \deg_F(\varphi_M^F(f))$, where $\deg_F$ is defined as $\deg_{M_F} \otimes \deg_{M^F}$.
\end{lemma}
\begin{lemma}[{Deletion map, \cite[Lemma 3.3]{BHM}}]\label{lem:deletion}
For a matroid $M$ on the ground set $E$, and an element $i \in E$, there is a deletion map $$\theta_i: A^\ast(M \setminus i) \rightarrow A^\ast(M), \;\; x_F \mapsto x_F + x_{F \cup \{i\}}$$ where a variable in the target is set to zero if its label is not a flat of $M$. 

Moreover, if $i$ is not a coloop, then $$\deg_{M \setminus i} = \deg_{M} \circ \;\theta_i.$$

If $i$ is a coloop, then $$\deg_{M \setminus i} = \deg_{M} \circ \;\alpha_M \circ \;\theta_i,$$ where the middle map in the composition denotes multiplication by $\alpha_M$ in the Chow ring of $M$.
\end{lemma}

\subsection{Valuative invariants}

For a matroid $M$ on the ground set $E$, the \emph{matroid polytope} $P_M \subset \mathbb{R}^E$ is defined to
be the convex hull of all $e_B$, where $B$ is a basis of $M$.

A function $f$ from the class of matroids on the ground set $E$ to an abelian group is called \emph{valuative} if it factors through the map that assigns to each matroid $M$ the indicator function of its matroid polytope $P_M$. That is, for any matroids $M_1, \ldots, M_k$ and integers $a_1, \ldots, a_k$ such that $\sum a_i 1_{P_{M_i}} = 0$, we require
that $\sum a_i f(M_i) = 0$.
\begin{proposition}[{\cite[Corollary 7.9]{EHL} or \cite[Theorem 5.4]{DF}}] \label{prop:linear} 
For a matroid $M$, the indicator function of its matroid polytope $P_M$ can be expressed as a linear combination of indicator
functions of matroid polytopes of Schubert matroids of the same rank. Moreover, Schubert matroids are realizable over any infinite field. 
\end{proposition}
\begin{remark}[{\cite[Remark 6.3]{LLPP2024}}]
Although the original statement applies to all matroids (possibly having loops), if we assume $M$ to be loopless, we can take every matroid in the linear combination to be loopless.
\end{remark}
Therefore, if we show that a valuative function is zero for realizable matroids over $\mathbb{C}$, then it is zero for all matroids, and we can extend geometric and combinatorial results from realizable to nonrealizable matroids.

\section{Tangent bundle and its Chern class}\label{sec:main result}

In this section, we define the cotangent and tangent classes of a general matroid by adapting the conormal exact sequence from classical algebraic geometry. We then use the tautological quotient class to compute explicit formulas for the total Chern class of the tangent bundle. Let $X_E$ be the permutahedral variety associated with the ground set $E = \{1, \dots, n\}$. For a matroid $M$ of rank $r$ on $E$, \cite[Definition 3.9]{BEST} defines the \emph{tautological quotient $K$-class} $Q_M \in K(X_E)$. If $M$ is realizable by $L \subseteq \mathbbm{k}^E$, then $Q_M$ is the class of a vector bundle $Q_L$ of rank $n - r$ on $X_E$ which admits a regular section whose vanishing locus is $W_L$. In this realizable case, denote the inclusion by $i:W_L \hookrightarrow X_E$, then $i^*Q_L$ is the normal bundle of $W_L$ in $X_E$. The following definition is motivated by the conormal exact sequence for the ideal sheaf $\mathcal{I}$ of $W_L$ in $X_E$:
\[
0 \to \mathcal{I}/\mathcal{I}^2 \to i^* \Omega_{X_E} \to \Omega_{W_L} \to 0.
\]

\begin{definition} \label{tangent bundle def}
For an arbitrary matroid $M$, we define 
\[
\widetilde{\Omega}_M := [\Omega_{X_E}] - [Q_M^\vee] \in K(X_E),
\]
and its dual:
\[
\widetilde{T}_M := [T_{X_E}] - [Q_M] \in K(X_E).
\]
In the realizable case, $i^*\widetilde{\Omega}_M = [\Omega_{W_L}]$ and $i^*\widetilde{T}_M = [T_{W_L}]$. 

For an arbitrary matroid $M$, we write $i_M:X_M \hookrightarrow X_E$. We set $\Omega_M = i_M^*(\widetilde{\Omega}_M)$ and $T_M = i_M^*(\widetilde{T}_M)$,  and call $\Omega_M$ and $T_M$ the \emph{cotangent bundle} and \emph{tangent bundle} of $M$, respectively. Note that these are $K$-classes in $K(M)$.
\end{definition}

\begin{definition} \label{Sidef}
For a matroid $M$ of rank $r$ and an integer $k > 0$, we define $$S_{k, M} = \sum_{F \text{ is a rank } r-k \text{ flat }} x_F \in A^1(M).$$
\end{definition}

The total Chern class of $T_M$ is $$c(T_M) = \frac{c\left(i^*_M(T_{X_E})\right)}{c\left(i^*_M(Q_M)\right)}.$$

By \cite[Proposition 13.1.2]{Toric}, the Chern class of the tangent bundle $T_{X_E}$ is given by $\prod_{\varnothing \subsetneq I \subsetneq E }(1+x_I)$. Upon restriction to $X_M$, only the terms corresponding to flats of $M$ survive. Therefore, \begin{equation}\label{eq:tangent bundle formula}
    c\left(i_M^*(T_{X_E})\right)= \prod_{\text{proper flats }F}(1+x_F) = \prod_{i=1}^{r-1}(1+S_{i,M}),
\end{equation} where the latter equality holds because two flats of the same rank are always incomparable.

The Chern class of $i^*_M(Q_M)$ is given as follows.

\begin{theorem}[{\cite[Appendix 3]{BEST} or \cite[Section 4.3]{Staircase}}]  \label{thm:stair}
For a matroid $M$ of rank $r$, the total Chern class of $i^*_M(Q_M)$ is given by $$\prod_{i=0}^{r-1}\frac{1}{(1+\alpha -\sum_{j=1}^i S_{j, M})}.$$
\end{theorem}

As a result, we have the following.

\begin{theorem}[Chern class of the tangent $K$-class] \label{Chris tangent formula}
For a matroid $M$ of rank $r$, the total Chern class of its tangent $K$-class $T_M$ is given by $$c(T_M) = \left(\prod_{i=1}^{r-1}(1+S_{i,M})\right)\cdot\left(\prod_{i=0}^{r-1}(1+\alpha-\sum_{j=1}^iS_{j,M})\right).$$
\end{theorem} 

\begin{corollary} \label{Only AG part}
For a matroid $M$ of rank $r$,  the total Chern class $c(T_M)$ lies in $\mathbb{Z}[\alpha_M, S_{1, M}, \ldots, S_{r-1, M}].$  
\end{corollary} 

\begin{remark}
In fact, for a rank $r$ matroid $M$, we have $$(1+S_{r-1, M})(1+\alpha-S_{1, M} - \cdots - S_{r-1, M}) = (1+\alpha - S_{1, M} - \cdots - S_{r-2, M}).$$ Thus $c(T_M)$ lies in $\mathbb{Z}[\alpha_M, S_{1, M}, \ldots, S_{r-2, M}]$, and we do not need $S_{r-1, M}$ in the statement of Corollary~\ref{Only AG part}. 
\end{remark}

\subsection{Hirzebruch--Riemann--Roch formula} \label{subsec:HRR} 

In this subsection, we use the valuative property of Euler characteristics to extend the classical Hirzebruch--Riemann--Roch formula from realizable matroids to arbitrary loopless matroids. 

The Todd class of a vector bundle $\mathcal{E}$ on $X$ is given by \[
\mathrm{td}(\mathcal{E})\;=\;\prod_i \frac{\alpha_i}{1 - e^{-\alpha_i}} 
\;\in\; A^*(X)\otimes\mathbb{Q},
\]
where \(\{\alpha_i\}\) are the Chern roots of \(\mathcal{E}\). The Todd class can be given explicitly as a formal power series in the Chern classes: $$\operatorname{td}(\mathcal{E}) = 1 + \frac{c_1(\mathcal{E})}{2} + \frac{c_1(\mathcal{E})^2 + c_2(\mathcal{E})}{12} + \cdots,$$ so the Todd class can be defined for any $K$-class.

For a smooth projective variety $X$ and any coherent sheaf $\mathcal{F}$, the Hirzebruch--Riemann--Roch formula implies $$\chi(\mathcal{F}) = \deg(\operatorname{ch}(\mathcal{F}) \cdot \operatorname{td}(T_X)),$$
where $T_X$ is the tangent bundle. Our goal is to relate this formula to Proposition \ref{prop: HRR} (Matroid Hirzebruch--Riemann--Roch). 

When $M$ is realizable by $L$, the Todd class in Proposition \ref{prop: HRR} then equals $\operatorname{td}(T_M)$. In this section, we will prove that this is true for arbitrary matroids. Our strategy is to use valuativity to extend from realizable to general matroids, and Proposition \ref{prop:valuative} is the key step.

\begin{proposition}\label{prop:valuative}

Let $X_E$ be the permutahedral variety. Fix an integer $r > 0$, a polynomial \(z \;\in\; \mathbb{Q}[x, y_1,y_2, \ldots, y_{r-1}] \), and a class $A \in A^*(X_E)$. For a rank $r$ matroid $M$ with ground set $E$, we set $$z_M = z(\alpha_M, S_{1, M}, \ldots, S_{r-1, M}) \in A^*(M).$$

Define
\[
\Phi_{z,A}: \{\text{rank } r \text{ matroids on }E\} \;\longrightarrow\; \mathbb{Q},
 \text{ by}\quad
M \;\longmapsto\; \deg_M\Bigl(i_M^*(A)\cdot z_M\Bigr).
\]

Then the map $\Phi_{z,A}$ is valuative for the set of rank $r$ matroids. 
\end{proposition}

The proof of Proposition \ref{prop:valuative} will be completed in Subsection \ref{proof of prop val}.

\subsection{Intersection numbers and flag counting functions}

To prove Proposition \ref{prop:valuative}, we must show that intersection numbers in the Chow ring behave valuatively. Our strategy in this subsection is to express products of the divisors $\alpha, \beta$, and $x_F$ entirely in terms of combinatorial flag counting functions. Because these counting functions are known to be valuative, this will establish the valuativity of the intersection numbers. Recall that in the Chow ring $A^*(M)$, the classes $\alpha$ and $\beta$ can be represented as sums of generators. For any fixed element $i \in E$, these classes satisfy the following identities:
\begin{equation} \label{eq:alpha-beta-def}
\alpha = \sum_{F \ni i} x_F, \quad \beta = \sum_{F \not\ni i} x_F.
\end{equation}

\begin{proposition} \label{prop:multiplication}
Let $F$ be a flat of rank $d$.
\begin{enumerate}
    \item For any $i \notin F$, we have $x_F \cdot \alpha = \sum_{G \supsetneq F, i \in G} x_F x_G$.
    \item For any $i \in F$, we have $x_F \cdot \beta = \sum_{G \subsetneq F, i \notin G} x_F x_G$.
\end{enumerate}
\end{proposition}

\begin{proof}
For (1), substituting the definition of $\alpha$ from \eqref{eq:alpha-beta-def} using an element $i \notin F$, we have $x_F \alpha = x_F (\sum_{G \ni i} x_G)$. The product $x_F x_G$ is non-zero only if $F \subseteq G$ or $G \subseteq F$. Since $i \in G$ and $i \notin F$, the case $G \subseteq F$ is impossible. Moreover, $G$ cannot be $F$ itself. Thus, the only surviving terms are those where $G \supsetneq F$. The proof for (2) is symmetric, choosing $i \in F$ and observing that $G \subsetneq F$ are the only comparable flats not containing $i$.
\end{proof}

\begin{lemma} \label{lemma:vanish}
Suppose $f = \prod_{F\in I} x_F$ is a degree $d$ monomial, where $I$ is a multiset of flats. If all flats in $I$ are strictly between two flats $F_1 \subsetneq F_2$ with $\operatorname{rk}(F_1) = d_1$ and $\operatorname{rk}(F_2) = d_2$, then $x_{F_1}fx_{F_2} = 0$ if $d_2-d_1\leq d$. (By convention we allow $F_1 = \varnothing$ or $F_2 = E$ and set $x_{\varnothing} = x_E = 1$).
\end{lemma}

\begin{proof}
Consider the matroid minor bounded by $F_1$ and $F_2$ (that is, contracting on $F_1$ and restricting to $F_2$). It is a rank $d_2 - d_1$ matroid, and everything with degree $\geq d_2 - d_1$ will vanish in the Chow ring. By the pullback map (Lemma \ref{lem:pullback}) of $F_1$ and $F_2$ we get $f = 0$.
\end{proof}

\begin{lemma}[Basic properties of $\alpha$ and $\beta$]
If $F$ is a flat of rank $d$, then $x_F\alpha^{r-d} = x_F \beta^{d} = 0$. Furthermore, $\deg(\alpha^{r-1}) = 1$ and $\deg(\beta^{r-1}) > 0$.
\end{lemma}
\begin{proof}
By Proposition \ref{prop:multiplication}, multiplying $x_F$ by $\alpha$ (resp. $\beta$) produces a sum of products $x_F x_G$ where $\operatorname{rk}(G) > d$ (resp. $\operatorname{rk}(G) < d$). Iterating this $r-d$ (resp. $d$) times forces the resulting flags to reach rank $r$ (resp. $0$), and the products vanish.

For the degree, Proposition \ref{prop:multiplication} implies $\alpha^{r-1}$ is a sum over maximal flags. Since for any flat $F$ and element $i \notin F$, there is a unique flat $G$ of rank $\operatorname{rk}(F)+1$ containing $F \cup \{i\}$, the iteration yields a unique maximal chain, hence $\deg(\alpha^{r-1}) = 1$. In contrast, for a flat $F$ and an element $i \in F$, there exists at least one flat $G$ of rank $\operatorname{rk}(F)-1$ covered by $F$ such that $i \notin G$, hence $\deg(\beta^{r-1}) > 0$.
\end{proof}

When $M$ is realizable (and $\mathbbm{k}$ is infinite), we can interpret this combinatorial result geometrically. The class $\alpha$ is the pullback of $\mathcal{O}(1)$. As a result, $\alpha^{r-d}$ can avoid everything of codimension $\geq d$, and the product $x_F\alpha^{r-d}$ is $0$ when $\operatorname{rk}(F) \geq d$. Moreover, $\deg(\alpha^{r-1}) = 1$. The geometric meaning of $\beta$ is not that straightforward, but it behaves in many ways like the dual of $\alpha$.

For an integer $d \ge 1$, we denote $[d] = \{1, \dots, d\}$.

\begin{definition} \label{definition of mathcal H}
For a rank $r$ matroid $M$ on the ground set $E$, denote by $\mathcal{H}$ the set of flags of flats. For a subset $I =\{i_1< \cdots < i_k\}\subseteq [r-1]$ write $\mathcal{H}_I$ for the set of flags of flats $\mathcal{G} = G_1\subsetneq G_2\subsetneq \cdots \subsetneq G_k$ with $\operatorname{rk}(G_t) = i_t$ for all $t$.

Let $\mathcal{F} = F_1\subsetneq \cdots \subsetneq F_l$ be a flag of flats, and let $I \subseteq \{1, \ldots, r -1\}$. We define the number $$N_{\mathcal{F}, I} = \# \{\mathcal{G} \in \mathcal{H}_I \text{ and }\mathcal{G} \sqcup \mathcal{F} \in \mathcal{H} \},$$ where the notation $\mathcal{G} \sqcup \mathcal{F} \in \mathcal{H}$ also implies $\mathcal{G}$ and $\mathcal{F}$ are disjoint. 

For an element $s \in E$ in the ground set, we define $$N_{\mathcal{F}, I}^{\not\ni s} = \# \{\mathcal{G} \in \mathcal{H}_I \mid \mathcal{G} = (G_1 \subsetneq \cdots \subsetneq G_k), s \notin G_k, \text{ and } \mathcal{G} \sqcup \mathcal{F} \in \mathcal{H} \}.$$
\end{definition}

\begin{lemma} \label{lemma: counting}
Let $M$ be a matroid of rank $r$ on a ground set $E$, and let $d$ be an integer such that $0 < d < r$. Suppose that $\mathcal{F} = \varnothing$ is the empty flag. Then for any element $s \in E$, we have
$$
\sum_{\substack{I \subseteq [d] \\ d \in I}} (-1)^{|I|}N_{\mathcal{F}, I}^{\not\ni s} = \sum_{I \subseteq [d]} (-1)^{|I|}N_{\mathcal{F}, I}.
$$
\end{lemma}
\begin{proof}

We proceed by induction on $d$. When $d = 1$, there is exactly one flat of rank 1 containing $s$. This corresponds to the empty set $I = \varnothing$ in the right-hand side, so the formula is correct. 

To compute the left-hand side, we first sum over all $G_k$, then subtract the case when $G_k$ contains $s$. For the part summing over all possible $G_k$, the value is $\sum_{I \subseteq [d], d \in I} (-1)^{|I|}N_{\mathcal{F}, I}$.

For a set $I = \{i_1,\dots,i_k\} \subseteq \{1,\dots,r - 1\}$ with $i_k = d$, and a number $j < k$, we define the number
$$N_{\mathcal{F}, I, j}^{s} = \# \{\mathcal{H}_{I} \ni \mathcal{G} = G_1\subsetneq G_2 \subsetneq \cdots \subsetneq G_k,  s \in G_{j+1}, s\notin G_j \text{, and $\mathcal{G} \sqcup \mathcal{F} \in \mathcal{H}$} \}.$$

In particular, \(j=0\) is allowed, and we use the conventions
\(G_0=\varnothing\) and \(i_0=0\).

We have $$\sum_{I \subseteq [d], d \in I} (-1)^{|I|}N_{\mathcal{F}, I}^{\not\ni s} = \sum_{I \subseteq [d], d \in I} (-1)^{|I|}N_{\mathcal{F}, I} - \sum_{\substack{I \subseteq [d], d \in I \\ 0 \leq j < |I|}} (-1)^{|I|}N_{\mathcal{F}, I, j}^{s}.$$

Denote $I = \{i_1, \dots, i_k\}$. Note that $s\notin G_j \text{ and } s\in G_{j + 1}$ is equivalent to $G_{j+1} \supseteq (G_j \cup \{s\})$, and there is a unique flat of rank $i_j + 1$ containing $G_j \cup \{s\}$. Therefore, if $i_{j+1} > i_j +1$, $$N_{\mathcal{F}, \{i_1, \ldots, i_k\}, j}^{s} = N_{\mathcal{F}, \{i_1, \ldots, i_j, i_j + 1, i_{j + 1}, \ldots, i_k\}, j}^{s}.$$

All terms will cancel out in the summation of $N_{\mathcal{F},I, j}^{s}$ except for the case $i_{k-1} +1 = i_k$ and $j=k -1$. Therefore, $$\sum_{\substack{I \subseteq [d], d \in I \\ 0 \leq j < |I|}} (-1)^{|I|}N_{\mathcal{F}, I, j}^{s} = \sum_{\substack{I = \{i_1, \ldots, i_k\}\subseteq [d]\\  d-1, d \in I, \text{ and } j = k - 1}} (-1)^{|I|}N_{\mathcal{F}, I, j}^{s}, $$ which is equivalent to $\sum_{I\subseteq [d-1], d-1 \in I} (-1)^{|I|+1}N_{\mathcal{F}, I}^{\not \ni s}$. 

By induction, this is $-\sum_{I \subseteq [d-1]} (-1)^{|I|}N_{\mathcal{F}, I}$, and we are done.
\end{proof}

\begin{lemma} \label{lemma: key valuative}
Let $M$ be a rank $r$ matroid on the ground set $E$, and let $0 \leq d < r$. We have
\[
\deg_M(\alpha^{r-1-d}\beta^{d}) = (-1)^d\sum_{I \subseteq [d]} (-1)^{|I|} N_{\varnothing, I},
\]
where $\varnothing$ is the empty flag.
\end{lemma}

\begin{proof}
We prove this by induction on $d$. This is true for $d = 0$ as both sides are 1. 

Expand one of the $\beta$'s. Choose an $s \in E$. The terms contributing to the degree are $$\sum_{\substack{s \notin F, \text{ and} \\ \text{$F$ has rank $d$}}} x_F \beta^{d-1}\alpha^{r-d-1}.$$

By the pullback map and pushforward map in Lemmas \ref{lem:pullback} and \ref{lem:pushforward} with respect to the flat $F$, $$\deg_M(x_F\beta^{d-1}\alpha^{r-d-1}) = \deg_F\big(\varphi^F_M(\beta^{d-1}\alpha^{r-d-1})\big) = \deg_{M^F}(\beta^{d-1}).$$

Now, sum over all rank $d$ flats that do not contain $s$. By induction and Lemma \ref{lemma: counting}, the summation is $$(-1)(-1)^{d-1}\sum_{I \subseteq [d], d \in I} (-1)^{|I|}N_{\varnothing, I}^{\not\ni s} = (-1)^d\sum_{I \subseteq [d]} (-1)^{|I|}N_{\varnothing, I}.$$
\end{proof}

\begin{theorem}[{Flag valuations, \cite[Theorem 6.2]{Valuative}}]\label{thm:flag-valuation}
Let $F_1 \;\subsetneq\;\cdots\;\subsetneq\; F_k$ be any fixed flag of subsets of $E$.  Define
\[
\widehat{\Phi}_{F_1,\dots,F_k}\colon \{\text{Matroids on } E\} \;\longrightarrow\; \mathbb{Z}
\]
by
\begin{equation}\label{eq:flag-val-basic}
\widehat{\Phi}_{F_1,\dots,F_k}(M)
=
\begin{cases}
1, & F_i \text{ is a flat for all }1\le i\le k,\\[6pt]
0, & \text{otherwise}.
\end{cases}
\end{equation}
Then $\widehat{\Phi}_{F_1,\dots,F_k}$ is valuative.

Moreover, for any rank-vector $(r_1,\dots,r_k)\in\mathbb{Z}^{\,k}$, define
\[
\widehat{\Phi}_{F_1,\dots,F_k}^{\,r_1,\dots,r_k}\colon \{\text{Matroids on } E\} \;\longrightarrow\; \mathbb{Z}
\]
by
\begin{equation}\label{eq:flag-val-rank}
\widehat{\Phi}_{F_1,\dots,F_k}^{\,r_1,\dots,r_k}(M)
=
\begin{cases}
1, 
& 
\begin{gathered}
F_i \text{ is a flat and } \mathrm{rk}(F_i)=r_i\\
\text{for all }1\le i\le k,
\end{gathered}
\\[10pt]
0, & \text{otherwise}.
\end{cases}
\end{equation}
Then $\widehat{\Phi}_{F_1,\dots,F_k}^{\,r_1,\dots,r_k}$ is also valuative. 
\end{theorem}
\begin{corollary} \label{Cor NFI valuation}
Let $\mathcal{F} = F_1 \;\subsetneq\;\cdots\;\subsetneq\; F_k$ be any fixed flag of subsets of $E$ and $I \subseteq \{1, \ldots, |E| - 1\}$.

Define 
\[
\widehat{N}_{\mathcal{F}, I}\colon \{\text{Matroids on } E\} \;\longrightarrow\; \mathbb{Z}
\]
by
\begin{equation}
\widehat{N}_{\mathcal{F}, I}(M)
=
\begin{cases}
N_{\mathcal{F}, I}, 
& 
\mathcal{F} \text{ is a flag of flats in } M,
\\
0, & \text{otherwise}.
\end{cases}
\end{equation}
Then $\widehat{N}_{\mathcal{F}, I}$ is valuative.
\end{corollary}
\begin{proof}
$\widehat{N}_{\mathcal{F}, I}$ is the sum of the functions in \eqref{eq:flag-val-rank} enumerating over all possible flats and ranks.
\end{proof}

\begin{proposition} \label{prop:degreeflagcounting}
Let $M$ be a matroid of rank $r$, and let
$$
\varnothing = F_0 \subsetneq F_1 \subsetneq \cdots \subsetneq F_k \subsetneq F_{k+1} = E
$$
be a flag of flats with $\operatorname{rk}(F_i) = r_i$. Let $d_0, d_1, \dots, d_{k+1}$ be nonnegative integers summing to $r-1$, such that $d_i > 0$ for all $1 \le i \le k$. Consider the monomial
$$
m = \beta^{d_0} x_{F_1}^{d_1} \cdots x_{F_k}^{d_k} \alpha^{d_{k+1}} \in A^{r-1}(M).
$$
Setting $\mathcal{F} = (F_1, \dots, F_k)$, we can express the degree of $m$ as
$$
\deg_M(m) = \sum_I c_I N_{\mathcal{F},I},
$$
where the coefficients $c_I \in \mathbb{Z}$ depend only on the sequences of exponents $(d_j)$ and ranks $(r_j)$.
\end{proposition}

\begin{proof}
Apply the pullback map $\varphi_M^{F_1} \colon A^*(M) \to A^*(M_{F_1}) \otimes A^*(M^{F_1})$ from Lemma~\ref{lem:pullback}. By Lemma~\ref{lem:pushforward}, the degree $\deg_M(m)$ factors as the product of the degrees in the respective Chow rings. For $0 \le i \le k$, we denote by $M_i$ the minor of $M$ bounded by $F_i$ and $F_{i+1}$ (that is, the minor obtained by contracting $F_i$ and restricting to $F_{i+1}$).

In the matroid $M^{F_1} = M_0$, the only nontrivial contribution comes from the terms $\beta$ and $x_{F_1}^{d_1 - 1}$. We expand the pullback $\varphi_M^{F_1}(x_{F_1}^{d_1-1}) = \left(-1 \otimes \alpha_{M^{F_1}} - \beta_{M_{F_1}} \otimes 1\right)^{d_1-1}$ using the binomial theorem. Because $M^{F_1}$ has rank $r_1$, only the term with degree $r_1 - 1$ in $A^*(M^{F_1})$ has a non-vanishing degree map. Since we already have a factor of $\beta^{d_0}$, the unique term providing the correct degree is
$$
(-1)^{d_1-1}\binom{d_1 - 1}{r_1-d_0 - 1} \deg_{M_0}\bigl(\alpha_{M_0}^{r_1 - d_0 - 1} \beta_{M_0}^{d_0}\bigr),
$$
and the corresponding tensor factor in $M_{F_1}$ is $\beta_{M_{F_1}}^{d_1-(r_1-d_0)} x_{F_2}^{d_2}\cdots x_{F_k}^{d_k} \alpha_{M_{F_1}}^{d_{k+1}}$.

Iterating this pullback process along the chain of flats $F_1, \dots, F_k$ yields
$$
\deg_M\bigl(\beta^{d_0}x_{F_1}^{d_1}\cdots x_{F_k}^{d_k}\alpha^{d_{k+1}}\bigr) \;=\; C \prod_{i=0}^{k} \deg_{M_i}\bigl(\alpha_{M_i}^{\,t_i}\beta_{M_i}^{\,s_i}\bigr),
$$
for integers $t_i, s_i$ and a combinatorial coefficient $C$ that depend only on the exponents $d_j$ and the ranks $r_j$ (but not on the specific matroid $M$).

By Lemma~\ref{lemma: key valuative}, each factor $\deg_{M_i}(\alpha_{M_i}^{t_i}\beta_{M_i}^{s_i})$ is a $\mathbb{Z}$-linear combination of $N_{\mathcal{F}_i,I}$ associated with the two-flat flag $\mathcal{F}_i = (F_i, F_{i+1})$ (where $I \subseteq \{r_i+1, \dots, r_{i+1}-1\}$). Consequently, the product $C \prod_{i=0}^{k} \deg_{M_i}(\alpha_{M_i}^{t_i}\beta_{M_i}^{s_i})$ expands into a linear combination of products $\prod_{i=0}^k N_{\mathcal{F}_i,I_i}$. Since each such product canonically equals $N_{\mathcal{F},I}$ for the full flag $\mathcal{F} = (F_1, \dots, F_k)$ with index set $I = \bigcup_i I_i$, we conclude that $\deg_M(m)$ is a $\mathbb{Z}$-linear combination of $N_{\mathcal{F},I}$, with coefficients depending exclusively on the $d_j$ and $r_j$.
\end{proof}

\begin{remark}
The proof may be shortened by applying \cite[Theorem 3.2]{EUR20}.
\end{remark}

\subsection{Valuativeness} \label{proof of prop val}

With the combinatorial flag counting identities established, we are now ready to prove Proposition \ref{prop:valuative} and, subsequently, the full matroid Hirzebruch--Riemann--Roch formula. Recall that $i_M$ is the inclusion $X_M \rightarrow X_E$.

\begin{proof} [Proof of Proposition \ref{prop:valuative}] 
The pullback map $i_M^\ast$ between Chow rings sends $x_F$ to $x_F$ if $F$ is a flat of $M$, and 0 otherwise. 

We can assume that $A$ and $z$ are monomials, and we expand $z_M$ as the sum of $x_{F_1}^{d_1}\cdots x_{F_k}^{d_k}\alpha^{d_{k+1}}$ over all flags of flats having the same sequence of ranks. This is the same as enumerating the flats in the definition of $N_{\mathcal{F}, I}$, where the flag of flats $\mathcal{F}$ is given by $A$. By Proposition \ref{prop:degreeflagcounting}, $\Phi_{z,A}$ is a linear combination of $\widehat{N}_{\mathcal{F}, I}$, and thus is valuative.  
\end{proof}

\begin{proposition}
For a matroid $M$ of rank $r$, the Todd class of the tangent bundle $\operatorname{td} (T_M)$ lies in $\mathbb{Q}[\alpha_M, S_{1, M}, \ldots, S_{r-2, M}].$
\end{proposition}
\begin{proof}
The Chern class of $T_M$ lies in $\mathbb{Z}[\alpha_M, S_{1, M}, \ldots, S_{r-2, M}]$.
\end{proof}
We have the following lemma.
\begin{lemma}[{\cite[Lemma 6.4]{LLPP2024}}]  \label{lem:chirestrictvalue}
For any class $R \in K(X_E)$, the map $M \mapsto \chi(i_M^\ast R)$ is valuative.
\end{lemma}

The restriction map $i_M^\ast: K(X_E) \rightarrow K(M)$ is surjective. For a rank $r$ matroid $M$ and a $K \in K(M)$, we can write $K = i_M^\ast(R)$ for some $R \in K(X_E)$. For a fixed $R \in K(X_E)$, we would like to show that both sides of the equation $$\chi(i_M^*(R)) = \text{deg}_M\big(\operatorname{ch}(i_M^*(R)) \cdot  \operatorname{td}(T_M)\big)$$ are valuative.

Since $i_M:X_M \rightarrow X_E$ is an open immersion, the Chern class, and hence the Chern character map, commutes with the pullback map $i_M^*$ (see \cite[Theorem 3.2(d)]{Fulton}). Hence, $\operatorname{ch}(i_M^*(R)) = i_M^*(\operatorname{ch}(R))$. Since $\operatorname{td}(T_M) \in \mathbb{Q}[\alpha, S_{1,M}, \ldots, S_{r-1, M}]$, the right-hand side is valuative by Proposition \ref{prop:valuative}.

The equation holds for realizable matroids of the same rank, so it can be extended to arbitrary matroids by Proposition \ref{prop:linear}. In particular, $$\chi(K) = \text{deg}_M\big(\operatorname{ch}(K) \cdot  \operatorname{td}(T_M)\big).$$

\begin{corollary} \label{Todd class formula}
For a matroid $M$ of rank $r$, the Todd class $\operatorname{Todd}_M = \operatorname{td}(T_M)$.
\end{corollary}

\begin{remark}
The factors in Theorem \ref{Chris tangent formula} are not the Chern roots since $T_{M}$ has rank $r-1$, not $2r-1$. Remarkably, the Todd class still equals $$\prod_{i=1}^{2r-1}\frac{t_i}{1-e^{-t_i}},$$ where $t_i = S_{i, M}$ for $0 < i < r$ and $t_{r+i} = \alpha - \sum_{j=1}^i S_{j, M}$ for $0\leq i < r$. This is because the Todd class formula does not depend on the rank of the bundle.
\end{remark}

\subsection{Serre duality for matroids}

In \cite{LLPP2024}, the Serre duality equation $\chi(\mathcal{E}) = (-1)^{r-1} \chi\bigl(\mathcal{E}^\vee \otimes \omega_{M}\bigr)$ is proved by extending the equation from the geometric case. We show the result from a different perspective.
\begin{corollary}
Let $M$ be a matroid of rank $r$, and let $\omega_{M} \in K(M)$ denote the $K$-class of the line bundle with first Chern class
$$
c_1(\omega_M) := -c_1(T_M) = -r\alpha + \sum_{i = 1}^{r-2} (r-i-1)S_{i, M}.
$$
Then for $\mathcal{E} \in K(M)$ we have
$$
\chi(\mathcal{E}) = (-1)^{r-1} \chi\bigl(\mathcal{E}^\vee \otimes \omega_{M}\bigr).
$$
\end{corollary}

\begin{proof}
First, recall a general fact about Todd classes: for any vector bundle $V$ with Chern roots $x_i$, $\operatorname{td}(V) = \prod \frac{x_i}{1-e^{-x_i}}$ and $\exp(-c_1(V)) = \prod e^{-x_i}$. Multiplying these together yields
$$
\operatorname{td}(V) \cdot \exp(-c_1(V)) = \prod \frac{x_i}{e^{x_i}-1} = \operatorname{td}(V^\vee) = \operatorname{td}(V)^\vee.
$$
Applying this to the tangent bundle $T_M$ and noting that $c_1(\omega_M) = -c_1(T_M)$, we obtain the relation $\operatorname{td}(T_M)\exp(c_1(\omega_M)) = \operatorname{td}(T_M)^\vee$. 

Substituting this identity into the Hirzebruch--Riemann--Roch formula
\begin{equation*} \label{eq:HRRinSerre}
    \chi\bigl(\mathcal{E}^\vee\otimes \omega_{M}\bigr) = \deg_M \big(\operatorname{ch}(\mathcal{E}^\vee) \cdot \operatorname{ch}(\omega_{M}) \cdot \operatorname{td}(T_M)\big), \tag{$\ast$}
\end{equation*}
and using the relations $\operatorname{ch}(\mathcal{E}^\vee) = \operatorname{ch}(\mathcal{E})^\vee$ and $\operatorname{ch}(\omega_M) = \exp(c_1(\omega_M))$, \eqref{eq:HRRinSerre} becomes:
$$
\chi\bigl(\mathcal{E}^\vee\otimes \omega_{M}\bigr) = \deg_M \big(\operatorname{ch}(\mathcal{E})^\vee \cdot \operatorname{td}(T_M)^\vee\big) = \deg_M \big( (\operatorname{ch}(\mathcal{E}) \cdot \operatorname{td}(T_M))^\vee \big).
$$
The dual operation $(\cdot)^\vee$ acts on the degree $k$ component of the Chow ring by multiplication by $(-1)^k$. Since the degree map $\deg_M$ extracts the top-degree component in $A^{r-1}(M)$, evaluating the dual of the product $\operatorname{ch}(\mathcal{E}) \cdot \operatorname{td}(T_M)$ pulls out a global sign of $(-1)^{r-1}$. Therefore,
$$
\chi\bigl(\mathcal{E}^\vee\otimes \omega_{M}\bigr) = (-1)^{r-1} \deg_M \big( \operatorname{ch}(\mathcal{E}) \cdot \operatorname{td}(T_M) \big) = (-1)^{r-1}\chi(\mathcal{E}).
$$
\end{proof}

\subsection{The Chow (Poincar\'e) polynomial}
For a matroid $M$ realized by $L$ over $\mathbb{C}$, \(W_L\) is a wonderful compactification obtained by an iterated sequence of blow-ups along smooth subvarieties, each of which is again a wonderful compactification of smaller rank. By inductively assuming that each blow-up center satisfies $h^{p,q}=0$ for all $p\neq q$, it follows that the final compactification $W_L$ also satisfies:
\[
h^{p,q}(W_L)=0 \quad \text{for all } p\neq q,
\]
i.e., its Hodge diamond is supported purely on the diagonal. Furthermore, the cycle map from the Chow ring to the cohomology ring continues to be an isomorphism. (For example, see \cite[Theorem 7.31]{Voisin}).  

For a matroid $M$ of rank $r$, one defines the Chow polynomial
\[
P_M(t) \,=\, \sum_{p=0}^{r-1} \dim A^{p}(M) \; t^p.
\]

If $M$ is realized by $L$, we have \[
\dim A^{p}(M) = \dim H^{2p}(W_L) = h^{p,p}(W_L) = (-1)^p \chi(W_L, \wedge^p \Omega_{W_L}).
\]

For an arbitrary matroid $M$ of a fixed rank $r$, we can expand $(-1)^p \chi(\wedge^p \Omega_M)$ by $\deg\left(\operatorname{ch}(\wedge^p \Omega_M) \cdot \operatorname{td}(T_M)\right)$. The Chern class of $\wedge^p \Omega_{M}$ can be derived from the Chern class of the tangent bundle, so it lies in the algebra generated by $\alpha, S_{1,M}, \ldots, S_{r-2, M}$. In particular, its Euler characteristic function is a polynomial in $\alpha, S_{1,M}, \ldots, S_{r-1, M}$, and is valuative. Moreover, the Chow polynomial is also valuative \cite[Section 8.4]{Valuative}. Since the equation $\dim A^{p}(M) = (-1)^p \chi(\wedge^p \Omega_M)$ holds for realizable matroids, it holds for arbitrary matroids.

\begin{theorem} \label{Chow polynomial formula}
The coefficients of the Chow polynomial are given by the formula: $$\dim A^{p}(M) = (-1)^p \chi(\wedge^p \Omega_M) = (-1)^p\deg\left(\operatorname{ch}(\wedge^p \Omega_M) \cdot \operatorname{td}(T_M)\right).$$
\end{theorem}

This shows how the dimension of the Chow ring can be derived from the tangent bundle class. For example, applying Theorem \ref{Chow polynomial formula} at $p = 1$ to a rank $r$ matroid $M$ yields:
\[
-\dim A^1(M) = \chi(T_M^\vee).
\]

The dimension of $A^1(M)$ equals the number of generators minus the number of linear relations, which is the number of proper flats minus the number of rank 1 flats, plus 1 (cf. Section \ref{subsec:Chow ring of a matroid}).

Furthermore, recall from Definition \ref{tangent bundle def} that:
\[
\chi(T_M^\vee) = \chi(i_M^*(T_{X_E})^\vee) - \chi(i_M^*(Q_M)^\vee).
\]

We can compute the first term, $\chi(i_M^*(T_{X_E})^\vee)$, using the following theorem.

\begin{theorem}\label{lem:chi-flatiszero}
Let $M$ be a matroid on ground set $E$, and let $F$ be a proper flat of $M$. Let $[\mathcal{X}_F] \in K(M)$ denote the $K$-class of the line bundle with $c_1([\mathcal{X}_F]) = x_F$. Then:
\[
\chi([\mathcal{X}_F]^{-1}) = \chi(\mathcal{O}(-x_F)) = 0.
\]
In particular, letting $n = |E|$, we have:
\[
\chi(i_M^*(T_{X_E})^\vee) = n - 1 - (\text{number of proper flats}).
\]
\end{theorem}
\begin{proof}
If $M$ is realizable by $L$, we can treat $x_F$ as a divisor on $W_L$. This yields the following short exact sequence of sheaves on $W_L$:
\[
0\longrightarrow\mathcal{O}(-x_F)\longrightarrow\mathcal{O}\longrightarrow\mathcal{O}_{x_F}\longrightarrow0.
\]

Because $W_L$ and the vanishing set of $x_F$ are rational varieties, we have $\chi(\mathcal{O})=\chi(\mathcal{O}_{x_F})=1$ in this geometric setting. It immediately follows that $\chi(\mathcal{O}(-x_F))=0$.

Next, fix a proper subset $\varnothing \subsetneq S \subsetneq E$ and let $[\mathcal{X}_S] \in K(X_E)$ denote the $K$-class of the line bundle with $c_1([\mathcal{X}_S])= x_S$. Consider the function on the set of matroids on a ground set $E$:
\[
f(M)=\chi(i_M^*([\mathcal{X}_S]^{-1})),
\]
which is valuative by Lemma \ref{lem:chirestrictvalue}. If $S$ is not a flat of $M$, then $i_M^*([\mathcal{X}_S]^{-1})$ is the trivial line bundle, and $f(M)=1$. If $S$ is a flat and $M$ is realizable, then $f(M)=0$ as shown above. 

The difference $f-\mathbbm{1}_{\{S\text{ is not a flat}\}}$ is also valuative and vanishes on all realizable matroids by Theorem~\ref{thm:flag-valuation}; therefore, it must vanish on all matroids. Thus, $\chi([\mathcal{X}_F]^{-1})=0$.

Recall that the Chern character of a class in the $K$-ring depends only on its total Chern class and its rank. The fact that the rank of $i_M^*(T_{X_E})^\vee$ is $n-1$ can be verified in the realizable case and extended to arbitrary matroids via a short valuative argument. By \eqref{eq:tangent bundle formula}, we have:
\[
c(i_M^*(T_{X_E})^\vee) = \prod_{\text{proper flats } F} (1 - x_F).
\]

Therefore, the Chern characters
\[
\operatorname{ch}(i_M^*(T_{X_E})^\vee) \quad \text{and} \quad\operatorname{ch}\left(\bigoplus_{\text{proper flats } F} [\mathcal{X}_F]^{-1}\right)\; = \sum_{\text{proper flats } F}\operatorname{ch}( [\mathcal{X}_F]^{-1})
\]
differ only in their degree-0 components, which represent the ranks of their respective classes. In particular, $\operatorname{ch}(i_M^*(T_{X_E})^\vee) - (n-1)= \sum_{\text{proper flats } F}\,(\operatorname{ch}( [\mathcal{X}_F]^{-1}) - 1)$. Hence,
\begin{align*}
    \chi(i_M^*(T_{X_E})^\vee) &= \deg(\operatorname{ch}(i_M^*(T_{X_E})^\vee)\cdot \operatorname{td}(T_M)) \\ 
    &= \left(n-1-\sum_{\text{proper flats } F} 1\right) \deg(\operatorname{td}(T_M))+ \sum_{\text{proper flats } F} \deg(\operatorname{ch}([\mathcal{X}_F]^{-1})\cdot \operatorname{td}(T_M)) \\
    &= n-1-\sum_{\text{proper flats } F} 1,
\end{align*}
where we use that $\deg(\operatorname{td}(T_M)) = 1$ and $\deg(\operatorname{ch}([\mathcal{X}_F]^{-1})\cdot \operatorname{td}(T_M)) = \chi([\mathcal{X}_F]^{-1}) = 0$. This completes the proof.

\end{proof}

This result directly yields the following corollary.

\begin{corollary}
Let $M$ be a matroid on ground set $E$. We have:
\[
\chi(i_M^*(Q_M)^\vee) = |E| - (\text{number of rank 1 flats}).
\]
\end{corollary}
\section{A geometric derivation via blow-ups}\label{sec:realizable}

In the previous section, we derived a product formula for $c(T_M)$. We now present a second, geometrically motivated expression for this same class. This approach begins with the realizable case, mirroring the blow-up construction of the De Concini--Procesi wonderful compactification, and then extends to all matroids via valuativity. While more computationally intensive, this method provides direct geometric insight and yields a recursive formula for the same Chern class, as stated in the following theorem.

\begin{theorem} \label{tangent bundle formula}
For each integer $n \geq 0$, define the polynomial $T_n(x,y_1,\dots,y_{n-1}) \in \mathbb{Z}[x, y_1, \ldots, y_{n-1}] \subset \mathbb{Z}[x, y_1, \ldots]$ such that $T_0 = 1$, and for $n>0$, $$T_n = (1 + x)^{n+1} + \sum_{\substack{1 \leq j \leq n - 1, \\ 0 \leq i \leq n + 1 - j}} \binom{n + 1 - j}{i}(T_{j-1})\big((1+y_j)(1-y_j)^i-1\big)(x-\sum_{k = 1}^{j - 1}y_k)^{n + 1 - i - j}.$$

For an arbitrary matroid $M$ of rank $r$, we write $T_{r-1}(M) = T_{r-1}(\alpha_M, S_{1, M}, \ldots, S_{r-2, M})$. Then for all matroids $M$, the total Chern class of the tangent bundle is given by $$c(T_M) = T_{r-1}(M).$$
\end{theorem}

\begin{remark}
For small values of $r$, one can verify that this recursive formula expands to the product formula given in Theorem~\ref{Chris tangent formula}. This also explains why there is no $S_{r-1, M}$ term, because the corresponding blow-up centers have codimension 1.
\end{remark}
The remainder of this section is dedicated to proving Theorem \ref{tangent bundle formula}.
\subsection{The realizable case}

Suppose $M$ is a rank $r$ matroid realized by $L \subseteq \mathbbm{k}^E$. The wonderful compactification $W_L$ is obtained by successively blowing up $\mathbb{P}^{r-1}$ along all linear subspaces corresponding to flats. At each step, we track the effect on Chern classes.

For any smooth variety $X$, we write $c(X)$ for the total Chern class of the tangent bundle, $c(T_X)$. 

Fix smooth projective varieties $X$ and $Y$ with $i:X\hookrightarrow Y$. Let $d$ be the codimension of $X$. When we blow up $Y$ along the center $X$, there is a blow-up diagram 

\begin{equation}\label{eq:blowup}
\begin{tikzcd}
\widetilde X \arrow[r, hookrightarrow, "j"] \arrow[d, "g"'] &
\widetilde Y \arrow[d,  "f"] \\
X \arrow[r, hookrightarrow, "i"] & Y
\end{tikzcd}\end{equation}

Let \(N=N_{X/Y}\) denote the normal bundle of \(X\) in \(Y\) (rank \(d\)). Then the exceptional divisor \(\widetilde X\subset\widetilde Y\) is isomorphic to \(\mathbb{P}(N)\), and the normal bundle \(N_{\widetilde X/\widetilde Y}\) is \(\mathcal{O}_{\mathbb{P}(N)}(-1)\).

\begin{proposition}[{\cite[Theorem 15.4]{Fulton}}]  \label{blowupchern}
With the above notation, and $\zeta =  c_1(\mathcal{O}_{\mathbb{P}(N)}(1))$,
    \[ c(\widetilde{Y}) - f^\ast(c(Y)) = j_{\ast}(g^\ast c(X) \cdot \gamma),\]
where 
\[ \gamma = \frac{1}{\zeta}\left[\sum_{i = 0}^d g^\ast c_{d-i}(N) - (1 - \zeta) \sum_{i = 0} ^d (1 + \zeta)^ig^\ast c_{d-i}(N)\right].\]

Here, the polynomial expression inside the brackets has no constant term with respect to $\zeta$, making this formal division $\frac{1}{\zeta}$ well-defined.
\end{proposition}
In order to compute $c(W_L)$ by the formula above, we need to compute the normal bundle of the blow-up center during the blow-up process. We will use the following proposition. 

\begin{proposition}[{\cite[B.6.10]{Fulton}}]  \label{blowupnormal}
Suppose $X \subsetneq Y \subsetneq Z$ is a sequence of closed regular embeddings. Let $\widetilde{Z} = \mathrm{Bl}_X Z$ and let $\widetilde{Y} = \mathrm{Bl}_X Y$ be the strict transform of $Y$. Then the embedding of $\widetilde{Y}$ into $\widetilde{Z}$ is a regular embedding, and its normal bundle is given by: $$N_{\widetilde{Y}/\widetilde{Z}} = \pi^\ast N_{Y/Z} \otimes \mathcal{O}(-E),$$ where $E$ is the exceptional divisor of the blow-up in $\widetilde{Y}$ and $\pi$ is the projection map $\widetilde{Y} \rightarrow Y$.
\end{proposition}

Let us compute $c(W_L)$ for a realization $L \subseteq \mathbbm{k}^E$ of small dimension to motivate the main theorem.

\begin{example}
Suppose $M$ is a rank 3 matroid realized by $L \subseteq \mathbbm{k}^E$. To construct $W_L$ we start from $\mathbb{P}L \cong \mathbb{P}^2$. Every rank 2 flat $F$ corresponds to a point $p_F$, and $W_L$ is the blow-up of $\mathbb{P}^2$ at all points $p_F$ for $F$ a rank 2 flat.

The total Chern class of \(\mathbb{P}^n\) equals \((1+\alpha)^{n+1}\), so for \(\mathbb{P}^2\) we have \(c_1=3\alpha\) and \(c_2=3\alpha^2\). Blowing up \(\mathbb{P}^2\) at \(p_F\) contributes \(-x_F\) to \(c_1\) and \(-x_F^2\) to \(c_2\). Therefore,
\[ c_1(W_L)=3\alpha-\sum_{F\ \text{rank }2} x_F = 3\alpha - S_{1, M},\qquad c_2(W_L)=3\alpha^2-\sum_{F\ \text{rank }2} x_F^2 = 3\alpha^2 - S_{1,M}^2. \]
\end{example}

\subsection{Chern class of wonderful compactification}
We compute the Chern class throughout the blow-up construction. The key steps involve tracking the normal bundles using Proposition~\ref{blowupnormal} and relating the exceptional divisors via Lemma~\ref{blowup pullback}.

Write \(W_L=X^{r-2}\to\cdots\to X^1\to X^0=\mathbb P^{r-1}\), where the map \(X^{i+1}\to X^i\) blows up the strict transforms of the linear subspaces \(\mathbb P L_F\) of dimension \(i\) (equivalently, flats of rank \(r-i-1\)). Consider blowing up a rank $r-i$ flat $F$ during the blow-up $X^{i} \rightarrow X^{i-1}$. Denote the blow-up center by $X_F$. By Proposition~\ref{blowupchern}, the difference of the Chern classes is given by     
\[ c(\widetilde{X}) - f^\ast(c(X^{i-1})) = j_{\ast}(g^\ast c(X_F) \cdot \gamma),\] where $\widetilde{X}$ is the result of the blow-up at $X_F$, and the notations of the maps coincide with the blow-up diagram \eqref{eq:blowup}.

The blow-up center $X_F$ is itself a wonderful compactification for the contracted matroid \(M_F\), which has rank \(i\). Its flats correspond to the flats \(G\supseteq F\) of \(M\), so we understand $c(X_F)$.

Initially, $N_{\mathbb{P}^{i - 1} / \mathbb{P}^{r - 1}} = \mathcal{O}(1)^{\oplus (r-i)}$. Applying Proposition~\ref{blowupnormal} iteratively shows that the normal bundle $N = N_{X_F / X^{i-1}} = \mathcal{O}(\alpha_{M_F})^{\oplus (r-i)} \otimes \mathcal{O}(-\sum_{F \subsetneq G \subsetneq E} x_G)$, where $x_G$ denotes the exceptional divisor coming from blowing up $G$. Equivalently, $$N_{X_F / X^{i-1}} = \mathcal{O}(\alpha_{M_F} - \sum_{F \subsetneq G \subsetneq E} x_G)^{\oplus (r-i)}.$$

We have described the Chern class of the normal bundle of $X_F$ in $A^\ast (X_F)$. To apply Proposition \ref{blowupchern}, we need to know $j_\ast g^\ast c(N)$.  One might expect that for a flat $G \supsetneq F$, the corresponding Chern class in $A^\ast (X_F)$ should be related to the corresponding Chern class in $W_L$.  The following lemma explains their relation.

\begin{lemma} \label{blowup pullback}
Let \( X \subsetneq Y \subsetneq Z \) be smooth varieties. Let \( \pi: Z' = \mathrm{Bl}_X Z \to Z \) denote the blow-up of \( Z \) along \( X \), and let \( Y' = \mathrm{Bl}_X Y \subset Z' \) be the strict transform of \( Y \). Denote by \( E_X^Y \subset Y' \) the exceptional divisor of the blow-up \( Y' \to Y \), and \( E_X^Z \subset Z' \) the exceptional divisor of the blow-up \( Z' \to Z \). Consider the blow-up \( \widetilde{Z} = \mathrm{Bl}_{Y'} Z' \), with exceptional divisor \( E_{Y'} \subset \widetilde{Z} \). Let \( j: E_{Y'} \hookrightarrow \widetilde{Z} \) be the inclusion, and \( g: E_{Y'} \to Y' \) the projection. Notice that the following diagram coincides with the blow-up diagram \eqref{eq:blowup}:

\[
\begin{tikzcd}
E_{Y'} \arrow[r, hookrightarrow, "j"] \arrow[d, "g"'] &
\widetilde{Z} \arrow[d,  "f"] \\
Y' = \mathrm{Bl}_X Y \arrow[r, hookrightarrow, "i"] &  Z'= \mathrm{Bl}_X Z
\end{tikzcd}
\]
Then, in\( A^1(E_{Y'}) \), we have:

\[
g^\ast (E_X^Y) = j^\ast f^\ast(E_X^Z).
\]
\end{lemma}
\begin{proof}
Observe that \(Y'\subset Z'\) and the exceptional divisor \(E_X^Z\subset Z'\) intersect transversely, and
\[
Y' \cap E_X^Z = E_X^Y.
\]

Since \(Y'\) meets \(E_X^Z\) transversely, the strict transform of \(E_X^Z\) under \(f\) meets the exceptional divisor \(E_{Y'}\) along the preimage of \(E_X^Y\). Hence \(j^*f^*(E_X^Z)=g^*(E_X^Y)\).
\end{proof}

Here, $Y' = X_F$ and $Z' = X^{i-1}$. Hence for a monomial $t = \alpha_{M _F}^{t_0} S_{1, M_F}^{t_1} \cdots S_{i-2, M_F}^{t_{i-2}} \in A^\ast(X_F)$, we have $$g^\ast(t) = j^{\ast}\big(\alpha_{M}^{t_0} S_{1, M}^{t_1} \cdots S_{i-2, M}^{t_{i-2}}\big).$$

Finally, note that $ j^\ast (x_F) = c_1(\mathcal{O}_{\mathbb{P}(N)}(-1)) = -\zeta$. By the projection formula, we have $$j_\ast (\zeta^l g^\ast(t)) = x_F (-x_F)^{l} \alpha_{M}^{t_0} S_{1, M}^{t_1} \cdots S_{i-2, M}^{t_{i-2}}.$$

At the stage $X^{i} \rightarrow X^{i-1}$, the centers corresponding to flats of rank $r-i$ are pairwise incomparable (and hence disjoint), so they may be blown up simultaneously. The summation of the $(-1)^l(x_F)^{l + 1}$ then becomes $(-1)^lS_{i,M}^{l + 1}$. 

To see exactly how the formula in Theorem \ref{tangent bundle formula} emerges from Proposition \ref{blowupchern}, we translate the geometric data of the blow-up into our polynomial variables. Let $n = r-1$. At the $j$-th stage of the construction (where we blow up the centers corresponding to flats), the normal bundle $N$ has rank $d = n + 1 - j$. 

The first Chern class of the associated line bundle is represented by $x - \sum_{k=1}^{j-1} y_k$. Therefore, the total Chern class of the normal bundle is $c(N) = (1 + x - \sum_{k=1}^{j-1} y_k)^{n+1-j}$. Extracting the $(d-i)$-th Chern class of $N$ directly yields the binomial term in our sum:
$$g^\ast c_{d-i}(N) = \binom{n + 1 - j}{i}\Big(x-\sum_{k = 1}^{j - 1}y_k\Big)^{n + 1 - i - j}.$$

The other terms in the summand map directly to the remaining components of Proposition \ref{blowupchern}:
\begin{itemize}
    \item The recursive factor $T_{j-1}$ corresponds to $g^\ast c(X_F)$, the Chern class of the blow-up center.
    \item The algebraic factor $\big((1+y_j)(1-y_j)^i-1\big)$ is the result of applying the pushforward $j_\ast$ to the $\zeta$ terms in $\gamma$, which converts the formal expressions involving the exceptional divisor class $\zeta$ into our polynomial variable $y_j$. 
\end{itemize}

Summing these contributions over all blow-up stages $j$ and indices $i$ recovers the recursive formula. Therefore, Theorem \ref{tangent bundle formula} holds for realizable matroids.

To complete the proof, we must show that this formula holds for all matroids. For a fixed class $A \in A^*(X_E)$, both the function $M \mapsto \deg_M\left(i_M^*(A) \cdot T_{r-1}(M)\right)$ and the function $M \mapsto \deg_M\left(i_M^*(A) \cdot c(T_M)\right)$ are valuative by Proposition \ref{prop:valuative}. Since the functions are equal on realizable matroids, we obtain $$\deg_M(i_M^*(A) \cdot T_{r-1}(M)) = \deg_M\left(i_M^*(A) \cdot c(T_M)\right)$$ for \emph{all} matroids $M$ and \emph{all} $A \in A^*(X_E)$. By Poincar\'e duality and the fact that $i_M^*$ is surjective, we conclude that $T_{r-1}(M) = c(T_M)$ for all matroids $M$, completing the proof.

\section{Big and nef divisors for matroids}\label{sec:nef}
In this section, we develop a combinatorial positivity theory for matroids, resulting in a formulation and partial verification of a matroid analogue to the Kawamata--Viehweg vanishing theorem. We proceed in three steps:

\begin{itemize}

    \item In Section~\ref{sec:nefstart}, we define combinatorially nef divisors by comparing three increasingly weak positivity conditions, establishing their relation to classical nefness and semiampleness in the realizable case.
    \item In Section~\ref{sec:big and nef}, we introduce the notion of combinatorially big and nef divisors. To mirror the classical characterization of these divisors, we construct the \emph{fake effective cone} and characterize big and nef classes via Theorem~\ref{fake effective cone}.
    \item In Section~\ref{sec:nef-kawamata}, we propose a matroid Kawamata--Viehweg vanishing conjecture. We then formulate a weakened version of this conjecture and prove it for rank $3$ matroids.
\end{itemize}
\subsection{Nef divisors} \label{sec:nefstart}
In classical algebraic geometry, nefness is defined via intersection numbers with curves. Because arbitrary matroids lack projective embeddings, we must define nefness combinatorially. We compare three increasingly weaker conditions for a divisor to be considered nef in the following theorem.

\begin{theorem}\label{thm:nef-definitions}
Let $M$ be a matroid of rank $r$ and write
\[
l=\sum_F l_F x_F \in A^1(M)
\]
for a degree-$1$ class (a divisor).  Consider the following properties of $l$:
\begin{enumerate}
  \item[(P1)] $l$ is the pullback of a nef divisor on \(X_E\); i.e., there exists \(\tilde l\in A^1(X_E)\) with \(\tilde l\) nef and \(i_M^*(\tilde l)=l\).
  \item[(P2)] For every flag \(\mathcal F=(F_1\subsetneq\cdots\subsetneq F_k)\) of proper flats (including the empty flag), there is an expression \(l=\sum_F c_F x_F\) with rational coefficients such that \(c_{F_i}=0\) for all \(i\) and \(c_F\ge0\) for every flat \(F\).
  \item[(P3)] For every flag \(\mathcal F=(F_1\subsetneq\cdots\subsetneq F_k)\) of proper flats (including the empty flag), there is an expression \(l=\sum_F c_F x_F\) with rational coefficients such that \(c_{F_i}=0\) for all \(i\) and \(c_F\ge0\) for every flat \(F\) with the property that \(\mathcal F\sqcup\{F\}\) is again a flag (i.e., $F$ can be inserted into the flag).
\end{enumerate}
Then \((\mathrm{P1})\implies(\mathrm{P2})\implies(\mathrm{P3})\), and both implications are strict.
\end{theorem}
We will later see in Proposition~\ref{prop: realizing nef} that, in the realizable case, (P2) and (P3) give realization-independent sufficient conditions for semiampleness and nefness, respectively.
\begin{proof}
If \(\tilde l=\sum_S c_S x_S\) is nef on \(X_E\) then by \cite[Proposition~4.4.3]{Matt_thesis} its coefficients
satisfy the positivity property required in (P2); pulling back to \(A^*(M)\) preserves that property,
so (P1)\(\Rightarrow\)(P2). Clearly (P2)\(\Rightarrow\)(P3).

We give brief counterexamples to show the implications are strict.

\medskip\noindent\textbf{(P3) $\not\Rightarrow$ (P2).}  Let \(M=U_{3,6}\).  Choose
\[
l=x_1+x_2+2x_{12}+x_{14}+x_{25}+x_{16}+x_{26}.
\]
A direct finite check shows that \(l\) satisfies (P3), but the coefficient
of \(x_{14}\) cannot be set to \(0\) without producing a negative coefficient elsewhere.

\medskip\noindent\textbf{(P2) $\not\Rightarrow$ (P1).}  Let \(M=U_{3,4}\) and
take
\[
l=2\alpha - x_{23}-x_{24}-x_{13}-x_{14} = x_1+x_2+2x_{12}=x_3+x_4+2x_{34}.
\]
Then \(l\) satisfies (P2): for any flag, at least one of the two forms $x_1+x_2+2x_{12}$ and $x_3+x_4+2x_{34}$ avoids the flag (conceptually, the divisor is ``basepoint free'').  However \(l\) is not the pullback of a nef divisor on \(X_E\): a nef divisor \(\tilde c=\sum_S c_S x_S\) on \(X_E\) must satisfy the submodularity relations \(c_{I\cup J}+c_{I\cap J}\le c_I+c_J\) with \(c_\varnothing=c_E=0\) by convention (see \cite[Proposition~2.2.6]{BES23}).  Checking these inequalities for the lift of \(l\) leads to a contradiction (one finds some \(c_T<0\)), so (P2) does not imply (P1).
\end{proof}
\begin{definition}
A class \(l\in A^1(M)\) is \emph{combinatorially nef} if it satisfies (P3). It is
\emph{combinatorially ample} if, moreover, for each flag, one can arrange the representing
coefficients to be strictly positive on the divisors that extend the flag. When the meaning
is clear we drop the adjective ``combinatorially'' and simply say ``nef'' or ``ample''.
\end{definition}

\begin{lemma}\cite[Proposition~4.5]{AHK18}
The combinatorially nef and combinatorially ample classes form nonempty convex cones in \(A^1(M)_\mathbb{R}\), and the ample cone is the interior of the nef cone.
\end{lemma}

\begin{example}
For any matroid \(M\) and any nonempty \(S\subseteq E\) the classes \(\alpha_S\) and \(\beta_S\)
are pullbacks from \(X_E\); in particular they satisfy (P1) and hence are nef.
\end{example}

\medskip

The following proposition explains the geometric meaning of (P2) and (P3) in the realizable case.

\begin{proposition} \label{prop: realizing nef}
Let \(M\) be realizable by \(L\subseteq \mathbbm{k}^E\) and let \(l\in A^1(M)\) be a divisor.
\begin{enumerate}
  \item If \(l\) satisfies (P2) then \(l\) is semiample on \(W_L\) (a positive multiple is basepoint free).
  \item If \(l\) satisfies (P3) then \(l\) is nef on \(W_L\).
\end{enumerate}
\end{proposition}
\begin{proof}
\begin{enumerate}
\item If \(l\) admits, for every flag \(\mathcal F\), an expression with nonnegative coefficients
vanishing on the flag, then we can choose a nonnegative rational representative of \(l\) that avoids any point of $W_L$ lying in the intersection of the divisors corresponding to \(\mathcal F\). The union of these intersections of divisors is equal to \(W_L\), and by finiteness of the set of flags this implies a positive multiple of \(l\) is basepoint free.
\item If \(C\subset W_L\) is an irreducible curve then \(C\) is contained in the intersection of divisors corresponding to some flag \(\mathcal F\); by (P3) we can ensure that the coefficients of \(l\) on divisors containing that intersection are nonnegative, so \(l\cdot C\ge0\).
\end{enumerate}
\end{proof}
\begin{remark}
For an arbitrary (and possibly noncomplete) fan $\Delta$, Gibney and Maclagan \cite{Nef} defined three analogous cones of divisors $G_\Delta \subseteq L_\Delta \subseteq F_\Delta$ for the toric variety corresponding to $\Delta$, establishing that the inclusions between them are strict in general. In relation to Theorem~\ref{thm:nef-definitions}, the condition (P1) corresponds to membership in $G_\Delta$, while (P3) corresponds to $L_\Delta$.

The converse implications in Proposition~\ref{prop: realizing nef} are false. Indeed, the conditions (P2) and (P3) depend only on the matroid \(M\), whereas the classical semiample and nef cones of the
wonderful compactification \(W_L\) may depend on the realization \(L\). For instance, in rank \(3\), the wonderful model is obtained from \(\mathbb P^2\) by blowing up the points corresponding to rank-two flats. The matroid does not record higher-degree conditions such as whether six of the points lie on a conic.
\end{remark}
\subsection{Big and nef divisors} \label{sec:big and nef}

In algebraic geometry, the following four statements are equivalent (and characterize big and nef divisors).
\begin{enumerate} 
    \item $D$ is nef, and it can (rationally) be written as ample + effective.
    \item There is an effective divisor $Z$ such that, for all $k$ sufficiently large, $D - Z/k$ is ample.
    \item $D$ is nef, and $\deg(D^{r-1}) > 0$.
    \item $D$ is nef and in the interior of the effective cone.
\end{enumerate}

The third statement yields a candidate definition of ``big and nef'' divisors for matroids.
\begin{definition} \label{def:big and nef}
For a matroid $M$ of rank $r$, a combinatorially nef divisor $l$ is \emph{combinatorially big and nef} if $\deg (l^{r-1}) > 0$.  Again, we omit the adjective ``combinatorially'' if the meaning is clear.
\end{definition}

It is tempting to define the \emph{combinatorially effective cone} for matroids, so that the big and nef divisors can be characterized by the equivalence of the following.  
\begin{goal}[the definition of ``combinatorially effective'' will have these properties be equivalent]    \label{equivalence of big and nef}
\quad
\begin{enumerate} 
    \item $D$ is nef, and it can (rationally) be written as ample + \emph{combinatorially effective}.
    \item There is a \emph{combinatorially effective} divisor $Z$ such that, for all $k$ sufficiently large, $D - Z/k$ is ample.
    \item $D$ is nef, and $\deg(D^{r-1}) > 0$.
    \item $D$ is nef and in the interior of the \emph{combinatorially effective} cone.
\end{enumerate}
\end{goal}
We refer to \cite[Section 2.2]{positivity} for the geometric background and intuition.
\begin{example} \label{Example:wrongcone}
The cone of divisors generated by nonnegative sums of $x_F$'s is not the correct one, i.e., does not have the property that (1)--(4)
in Goal~\ref{equivalence of big and nef} are equivalent.

Consider the matroid $U_{3, n}$ for $n > 4$. The class $D = 2\alpha - x_{12} - x_{34} = \alpha_{12} + \alpha_{34}$ is nef, and $\deg(D^2) = 2^2 - 1 - 1 = 2$. However, $x_i$ for $i > 4$ cannot appear. Here, if we pick $Z = x_5$, then $D - \epsilon Z$ cannot be written as a nonnegative sum of ${x_F}$ for any positive $\epsilon$. Thus, (3) does not imply (4) in this setting.
\end{example}

Here is a proposal of a definition for which Goal~\ref{equivalence of big and nef} is achieved.

\begin{definition}
Let $M$ be a matroid of rank $r$.  The \emph{fake effective cone} $\mathcal F_M$
is the closed convex cone of classes
\[
\mathcal{F}_M \;=\; \{\, D\in A^1(M)_\mathbb{R} \;:\; \deg( D\cdot \ell_1\cdots\ell_{r-2} ) \ge 0
\ \text{for all combinatorially nef }\ell_1,\dots,\ell_{r-2}\,\}.
\]
A class $D\in\mathcal F_M$ is called \emph{fake effective}.
\end{definition}
\begin{theorem}\label{fake effective cone}
With the fake effective cone $\mathcal F_M$ in place, the four properties (1)--(4)
in Goal~\ref{equivalence of big and nef} are equivalent: for a divisor $D$ the
following are equivalent.
\begin{enumerate}
  \item $D$ is combinatorially nef, and it can (rationally) be written as $D=A+Z$ with $A$ combinatorially ample and $Z\in\mathcal F_M$.
  \item There exists $Z\in\mathcal F_M$ such that $D-\tfrac{1}{k}Z$ is combinatorially ample for all sufficiently large integers $k$.
  \item $D$ is combinatorially nef and $\deg(D^{\,r-1})>0$.
  \item $D$ is combinatorially nef and lies in the interior of $\mathcal F_M$.
\end{enumerate}
\end{theorem}

To prove this, we will apply a version of the reverse Khovanskii--Teissier inequality, and we will need the notion of Lorentzian polynomials. For definitions and properties on Lorentzian polynomials, see \cite{BH20}.

\begin{proposition} [{\cite[Theorem 8.9]{AHK18}}]
Let $M$ be a matroid of rank $r$, and $\ell_1, \ldots, \ell_n$ be combinatorially ample divisors. Then the function $$f(x_1, \ldots, x_n) = \frac{1}{(r-1)!}\deg\big((x_1\ell_1 + \cdots + x_n \ell_n)^{r-1} \big)$$ is strictly Lorentzian. The polynomial $f$ is called the \emph{volume polynomial}.
\end{proposition}

In particular, for nef divisors $\ell_1, \ldots, \ell_{r-1}$, we have $\deg(\ell_1\cdots \ell_{r-1}) \geq 0$.

\begin{theorem}[{\cite[Theorem 2.5]{JJ}}] \label{thm:rkt}
Let $f$ be a degree $d$ Lorentzian polynomial with $n$ variables. Then for any $x \in \mathbb{R}_{\geq 0}^n$ and for any $\alpha, \beta, \gamma \in \mathbb{N}^n$ satisfying $\alpha = \beta + \gamma$, $|\alpha| \leq d$, we have $$f(x)\partial^\alpha f(x) \leq c_{\alpha, \beta, \gamma, d} \;\;\partial ^\beta f(x) \partial ^ \gamma f(x) $$ for a positive constant $c_{\alpha, \beta, \gamma, d}$ determined by $\alpha, \beta, \gamma$, and $d$. \end{theorem}

(The constant is described in \cite{JJ}). Differentiating the volume polynomial with respect to $x_i$ corresponds to intersecting with $\ell_i$, by the chain rule. In particular, using the fact that nef divisors are the limit of ample divisors, we have the following.

\begin{corollary} \label{weak ineq}
Let $M$ be a matroid of rank $r$, and let $m, n, k$ be nonnegative integers such that $r - 1 \geq k = m + n$. For any combinatorially nef divisors $\ell, \ell_1, \ldots, \ell_k$, there exists a constant $c_{m, n, r} > 0$ that only depends on $m, n, r$, such that $$\deg(\ell^{r-1}) \deg(\ell^{r-1-k}\ell_1\cdots\ell_{k}) \leq c_{m, n, r}\deg(\ell^{r-1-m}\ell_1 \cdots \ell_{m})\deg(\ell^{r-1-n}\ell_{m+1}\cdots \ell_{k}).$$
\end{corollary}
We remark that for a smooth projective variety $X$, after replacing the notion `combinatorially nef' with `nef', the inequality in Corollary \ref{weak ineq} is called the reverse Khovanskii--Teissier (rKT) inequality for a better constant $c_{m, n, r} = \binom{m+n}{m}$ (see \cite[Theorem 1.2]{JJ}).

\begin{proof}[Proof of Theorem \ref{fake effective cone}]
We prove the equivalences by a series of implications.

\medskip\noindent\textbf{(4)\(\Rightarrow\)(1).} 
Pick an ample $A$. Since $D$ is inside the interior of $\mathcal{F}_M$, $D - \epsilon A$ is fake effective for $\epsilon$ small enough.

\medskip\noindent\textbf{(1)\(\Rightarrow\)(4).} 
Write $D = A + Z$ where $A$ is ample and $Z$ is fake effective. For a divisor $B$ and a small enough $\epsilon$, $D - \epsilon B = (A - \epsilon B) + Z$ is (ample) + (fake effective), which is fake effective.

\medskip\noindent\textbf{(1)\(\Rightarrow\)(2).} Since $D = A + Z$, we have $kD-Z = (k-1)D+A$. Because $D$ is nef and $A$ is ample, this sum is ample.

\medskip\noindent\textbf{(2)\(\Rightarrow\)(1).} $D - \frac{1}{k}Z$ is ample immediately implies $D$ can be written as (ample) + (fake effective).

\medskip\noindent\textbf{(3)\(\Rightarrow\)(4).} Suppose $D$ satisfies (3), and let $A$ be an ample divisor.

The inequality in Corollary \ref{weak ineq} tells us that for any nef divisors $\ell_2, \ldots, \ell_{r-1}$:
$$\deg (D^{r-1}) \deg(A\ell_2 \cdots \ell_{r-1}) \leq c  \deg(D^{r-2} A) \deg (D \ell_2 \cdots \ell_{r-1}).$$

After scaling $A$ (for example letting $\deg (D^{r-2} A) = 1$), the inequality means $\deg (D \ell_2 \cdots \ell_{r-1}) \geq C_2 \deg (A \ell_2 \cdots \ell_{r-1})$ is true for all nef $\ell_2, \ldots, \ell_{r-1}$. Therefore, $D-\epsilon A$ is inside the fake effective cone for small enough $\epsilon$, and $D$ satisfies (4).

\medskip\noindent\textbf{(4)\(\Rightarrow\)(3).} Suppose $D$ is nef and in the interior of the fake effective cone. Then, after scaling, there exists an ample divisor $A$ such that $(D-rA)$ is fake effective and $(D+A)$ is ample, where $r$ is the rank of the matroid. By definition $\deg((D-rA)(D+A)^{r-2}) \geq 0$. Therefore, 

\begin{align*}
0 \leq \deg\left((D-rA)(D+A)^{r-2}\right)      &= \deg\left((D-rA)(D^{r-2} + (r-2)D^{r-3}A +  \cdots + A^{r-2})\right) \\
 &= \deg(D^{r-1} - 2D^{r-2}A - \cdots - rA^{r-1})
\end{align*}

This means $\deg(D^{r-1}) > 0$, and $D$ satisfies (3).

\end{proof}
\begin{remark}
Similarly, for a smooth projective variety $X$ of dimension $r-1$, we can define its \emph{fake effective cone} 
\[
\mathcal{F} \;=\; \{\, D\in A^1(X)_\mathbb{R} \;:\; \deg( D\cdot \ell_1\cdots\ell_{r-2} ) \ge 0
\ \text{for all nef }\ell_1,\dots,\ell_{r-2}\,\}.
\]

Every $\mathbb{Q}$-effective divisor is inside $\mathcal{F}$, but $\mathcal{F}$ can be strictly larger than the actual effective cone. For example, let $X$ be the blow-up of $\mathbb{P}^3$ at four general points, and consider the divisor $D = H - E_1-E_2-E_3-E_4$, where $H$ is the pullback of the hyperplane class and the $E_i$ are the exceptional divisors. We claim that $D$ is in $\mathcal{F}$ but is not $\mathbb{Q}$-effective.

If $mD$ is effective for a positive integer $m$, we have a homogeneous degree $m$ polynomial $f(x_1, x_2, x_3,x_4)$ having multiplicity at least $m$ at each of the four points, and we may assume that they are the coordinate vertices. For $f$ to have multiplicity $m$ at $[1:0:0:0]$, $f(1, x_2, x_3, x_4)$ must consist only of terms of degree $\ge m$, and hence $x_1$ does not appear in $f$ at all. Applying this to the other three vertices forces $x_2, x_3,$ and $x_4$ to be absent, which is impossible. Therefore, $D$ is not $\mathbb{Q}$-effective.

However, $\deg(D \cdot \ell_1 \cdot \ell_2) \ge 0$ for any two nef divisors $\ell_1, \ell_2$. Write $\ell_k = a_k H - \sum_{i=1}^4 b_{k,i} E_i$. Because $\ell_k$ is nef, its intersection with a general line in $X$ and with a line in $E_i \cong \mathbb{P}^2$ must be nonnegative, which implies $a_k \ge 0$ and $b_{k,i} \ge 0$. Furthermore, $\ell_k$ must intersect the strict transform of the line through any two exceptional points nonnegatively, yielding $a_k \ge b_{k,i} + b_{k,j}$ for all $i \neq j$. 

The intersection product evaluates to $\deg(D \cdot \ell_1 \cdot \ell_2) = a_1 a_2 - \sum_{i=1}^4 b_{1,i} b_{2,i}$. Assuming that $b_{1,1} \ge b_{1,2} \ge b_{1,3} \ge b_{1,4} \ge 0$, we can bound the sum:
$$
\sum_{i=1}^4 b_{1,i} b_{2,i} \le b_{1,1}(b_{2,1}+b_{2,2}) + b_{1,3}(b_{2,3}+b_{2,4}) \le (b_{1,1}+b_{1,3})a_2 \le a_1 a_2.
$$

Thus, the intersection is nonnegative, and $D$ is fake effective.

For a smooth projective variety $X$, the equivalence of (1), (2), (3), and (4) in Goal \ref{equivalence of big and nef} still holds for the fake effective cone via the same proof. Therefore, the interior of the fake effective cone and the interior of the effective cone coincide when restricting to the nef case.

Hence, the ``correct'' definition of the effective divisors for matroids should strictly contain the cone generated by the $x_F$'s (cf. Example \ref{Example:wrongcone}), and be contained inside the fake effective cone.
\end{remark}

\subsection{Matroid Kawamata--Viehweg vanishing conjecture}\label{sec:nef-kawamata}

Let $X$ be a smooth projective variety of dimension $d$ over $\mathbb{C}$, and let $\mathcal{L}$ be a line bundle on $X$. If $\mathcal{L}$ is big and nef, the Kawamata--Viehweg vanishing theorem tells us $H^i(X, \mathcal{L}^{-1}) = 0$ for $i < d$. In particular, $$(-1)^{d} \chi(X, \mathcal{L}^{-1}) \geq 0.$$

This motivates the following conjecture.

\begin{conjecture}[Matroid Kawamata--Viehweg Vanishing] \label{KV-vanishing}
Let $M$ be a matroid of rank $r$, and let $\ell$ be a big and nef divisor. If $[\mathcal{L}] \in K(M)$ is the $K$-class of the line bundle with $c_1([\mathcal{L}]) = \ell$, then $$(-1)^{r-1}\chi([\mathcal{L}]^{-1}) \geq 0.$$
\end{conjecture}

By Proposition \ref{prop: realizing nef}, the conjecture automatically holds when $M$ is realizable over $\mathbb{C}$. Furthermore, it is shown in \cite{KVsurface} that the Kawamata--Viehweg vanishing theorem holds for rational surfaces of characteristic $p$. Hence, Conjecture \ref{KV-vanishing} is true when $r=3$ and the matroid is realizable. This gives us more evidence that the conjecture could be true for general matroids.

In \cite{EL}, the authors proved the following result. (In fact, they proved something stronger).

\begin{theorem}[{\cite[Theorem 1.5]{EL}}] \label{h vector} Let $M$ be a matroid and $[\mathcal{L}] = \prod_F [\mathcal{L}_F]^{c_F} \in K(M)$ be the $K$-class of a line bundle where the $c_F$ are nonnegative integers (cf. Section \ref{subsec: K Chern}). Let $d$ be the numerical dimension of $c_1([\mathcal{L}])$ (i.e., the largest integer $t$ such that $c_1([\mathcal{L}])^t \neq 0$ in the Chow ring). Then $$(-1)^d\chi([\mathcal{L}]^{-1}) \geq 0.$$ \end{theorem}

We remark that the more general statement for Theorem~\ref{h vector} has recently been proved by Eur, Fink, and Larson for divisors that are pullbacks of nef line bundles from the permutahedral variety $X_E$ (\cite[Theorem B]{Vanishhindex}). Consequently, Conjecture~\ref{beta conjecture} and several derived statements in Section~\ref{sec:Equations} become automatic. We include a self-contained treatment here, and we also analyze the more general combinatorially nef and big-and-nef cases, which remain open beyond the scope of \cite{Vanishhindex}.

One might expect Theorem \ref{h vector} to be true for all $K$-classes of the line bundles $[\mathcal{L}]$ when $c_1([\mathcal{L}])$ is nef. However, the inequality can fail for general classes even when $c_1([\mathcal{L}])$ satisfies (P2).  We record a family of counterexamples.

\begin{example}
Let \(M=U_{3,2k}\) and write \(I=\{1,\dots,k\}\), \(J=\{k+1,\dots,2k\}\).  Consider
\[
\ell \;=\; k\alpha \;-\;\sum_{i\in I,\,j\in J} x_{ij}.
\]
One checks that \(\ell^2=0\), so the numerical dimension of \(\ell\) is \(1\). Letting $[\mathcal{L}]$ be the $K$-class of the line bundle with $c_1([\mathcal{L}]) = \ell$,  one obtains
\[
\chi([\mathcal{L}]^{-1}) \;=\; \deg\bigl((1-\ell)\,\operatorname{Todd}_M\bigr)
    \;=\;
    \frac{k^2-3k+2}{2}.
\]
Hence for \(k>3\) we have \((-1)^1\chi([\mathcal{L}]^{-1})=-\tfrac{k^2-3k+2}{2}<0\), so the
expected sign property fails for this divisor even though \(\ell\) satisfies (P2).
\end{example}

Fix a matroid $M$ of rank $r$. If $\ell$ is a nef divisor, then $\alpha + \ell$ is big and nef. Let $[\mathcal{L}]$ be the $K$-class of the line bundle with $c_1([\mathcal{L}]) = \ell$. Conjecture \ref{KV-vanishing} predicts that $$(-1)^{r-1}\chi([\mathcal{L}]^{-1}[\mathcal{L}_E]^{-1}) \geq 0.$$

Recalling \eqref{Euler char exceptional}, we may compute $$\chi([\mathcal{L}]^{-1}[\mathcal{L}_E]^{-1}) = \deg\left(\zeta_M([\mathcal{L}]^{-1}) \zeta_M([\mathcal{L}_E]^{-1})(1+\alpha+\alpha^2+ \cdots)\right) = \deg(\zeta_M([\mathcal{L}]^{-1})),$$ as $\zeta_M([\mathcal{L}_E]^{-1}) = 1 - \alpha$.

This motivates the following statement. 

\begin{conjecture}[weaker version of Matroid Kawamata--Viehweg Vanishing] \label{KV vanishing weak}
Let $M$ be a rank $r$ matroid, let $\ell$ be a nef divisor, and let $[\mathcal{L}]$ be the $K$-class of the line bundle with $c_1([\mathcal{L}]) =\ell$. Then $$(-1)^{r-1}\deg\left(\zeta_M([\mathcal{L}]^{-1})\right) \geq 0.$$
\end{conjecture}

\begin{definition} \label{def: P class}
Fix a matroid $M$ of rank $r$. For $0 \leq i \leq r- 1$, we define the sets
\begin{enumerate}
    \item $N_i \subset A^i(M)$ to be the set of elements $x \in A^i(M)$ that can be expressed as a product of $i$ nef divisors.
    \item $P_i \subset A^i(M)$ to be the set of elements $x \in A^i(M)$ that can be expressed as a nonnegative linear combination of elements in $N_i$.
\end{enumerate} 

We then define a set of divisors $\mathcal{P} \subset A^1(M)$ to consist of those divisors $\ell$ such that $(-1)^i$ times the degree-$i$ part of $\zeta_M([\mathcal{L}]^{-1})$ lies in $P_i$ for all $i$, where $[\mathcal{L}]$ is the $K$-class of the line bundle with $c_1([\mathcal{L}]) = \ell$.
\end{definition}

\begin{theorem} \label{thm:subcase of the conjecture}
The set $\mathcal{P}$ satisfies the following properties.
\begin{enumerate}
    \item $\mathcal{P}$ is contained in the set of nef divisors.
    \item $\mathcal{P}$ is closed under addition.
    \item Conjecture \ref{KV vanishing weak} holds for divisors in $\mathcal{P}$.
\end{enumerate}
\end{theorem}
\begin{proof}
\begin{enumerate}
    \item  
    
    Suppose that $\ell \in \mathcal{P}$ and $[\mathcal{L}]$ is the corresponding $K$-class of the line bundle. The degree-1 part of $\zeta_M([\mathcal{L}_F])$ is $\alpha_F$, and every divisor is a $\mathbb{Z}$-linear combination of $\alpha_F$'s. Therefore, the degree-1 part of $\zeta_M([\mathcal{L}])$ is precisely $\ell = c_1([\mathcal{L}])$. Thus, $\ell$ is a nonnegative linear combination of nef divisors, which is nef.
    \item If $\ell_1, \ell_2 \in \mathcal{P}$ with $[\mathcal{L}_1], [\mathcal{L}_2]$ being the corresponding $K$-classes of the line bundles, then $\zeta_M([\mathcal{L}_1]^{-1}[\mathcal{L}_2]^{-1}) = \zeta_M([\mathcal{L}_1]^{-1}) \cdot \zeta_M([\mathcal{L}_2]^{-1})$, and $\ell_1 + \ell_2 \in \mathcal{P}$. 
    \item If $\ell \in \mathcal{P}$ with $[\mathcal{L}]$ being the corresponding $K$-class of the line bundle, then $(-1)^{r-1}\deg\left(\zeta_M([\mathcal{L}]^{-1})\right)$ is a nonnegative linear combination of products of nef divisors, and therefore nonnegative.
\end{enumerate}
\end{proof}

It is natural to ask the following question.

\begin{question}
When does a nef divisor belong to $\mathcal{P}$?
\end{question}
We answer the question for rank 3 matroids. Let $M$ be a rank 3 matroid, and $\ell$ be a nef divisor with the corresponding $K$-class of the line bundle $[\mathcal{L}]$. Since the degree-1 part of $\zeta_M([\mathcal{L}])$ is $\ell$, we only need to check the degree-2 part, and $\ell \in \mathcal{P}$ if and only if Conjecture \ref{KV vanishing weak} holds for $\ell$.
\begin{theorem} \label{thm: rank3KV}
Let $M$ be a rank 3 matroid. Then $\mathcal{P}$ is exactly the set of nef divisors. Equivalently, Conjecture \ref{KV vanishing weak} holds for rank 3 matroids.
\end{theorem}
The rest of the section is devoted to proving Theorem \ref{thm: rank3KV}.
\begin{proposition}
Let $M$ be a rank 3 matroid, and let $\ell \in A^1(M)$ be a divisor with $[\mathcal{L}]$ the corresponding $K$-class of the line bundle. Then $$\zeta_M([\mathcal{L}]) = 1 + \ell + \frac{\ell(\ell + \alpha - S_{1,M})}{2}$$ (cf. Definition \ref{Sidef}).
\end{proposition}

\begin{proof}
The formula is verified directly for $[\mathcal{L}_F]$ and $[\mathcal{L}_F]^{-1}$. 
Since it equals $\exp((1 + (\alpha - S_{1,M})/2) \ell)$ and is multiplicative, 
it extends to all integral linear combinations of $\pm \alpha_F$, 
which generate the group of divisors.
\end{proof}

Thus, it remains to show that $\deg\left(\ell(\ell - \alpha +  S_{1,M})\right) \geq 0$ for a nef divisor $\ell$.

\begin{lemma} \label{lem: big and nef nonzero divisor}

Let $M$ be a matroid of rank $r$, and suppose $\ell,\ell_1,  \ldots, \ell_k$ are combinatorially nef divisors with $1\leq k \leq r - 2$. If $\deg(\ell^{r-1}) > 0$, then $\ell\ell_1 \cdots \ell_k = 0$ implies $\ell_1 \cdots \ell_k = 0$.
\end{lemma}
\begin{proof}
If $\ell_1 \cdots \ell_k \neq 0$, we can find combinatorially nef divisors $\ell_{k+1}, \ldots, \ell_{r-1}$ such that $\ell_1 \cdots\ell_{r-1} \neq 0$. By Corollary \ref{weak ineq}, there is a constant $c>0$ such that $$\deg(\ell^{r-1}) \deg(\ell_1\cdots\ell_{r-1}) \leq c \deg(\ell^{r-1-k}\ell_1 \cdots \ell_k)\deg(\ell^{k}\ell_{k+1}\cdots \ell_{r-1}).$$ But the left-hand side is positive and the right-hand side is zero, which is impossible.
\end{proof}
\begin{lemma}
Let $M$ be a rank 3 matroid, and let $\ell \in A^1(M)$ be a nonzero nef divisor. Then there exist a positive integer $a$, a nonnegative integer $n$, positive integers $b_1, \ldots, b_n$, and rank 2 flats $F_1, \ldots, F_n$ such that $$\ell = a\alpha - \sum_{i=1}^n b_i x_{F_i}.$$
\end{lemma}

\begin{proof}
The divisor $\alpha$ is big and nef. By Lemma \ref{lem: big and nef nonzero divisor}, we have $\deg(\alpha \cdot \ell) \in \mathbb{Z}_{>0}$. For two rank 1 flats $F_i, F_j$, the difference $x_{F_i} - x_{F_j}$ can be expressed as a linear combination of $x_F$ where $F$ ranges over rank 2 flats. Hence, $\ell - \deg(\alpha \cdot \ell) \alpha$ is a linear combination of $x_F$ with $F$ of rank 2. Setting $a = \deg(\alpha \cdot \ell)$, we obtain $$\ell = a\alpha - \sum_{i=1}^n b_i x_{F_i}.$$ 

Moreover, $\deg(\ell \cdot x_{F_i}) = b_i\geq 0$, which completes the proof.
\end{proof}

\begin{proof}[Proof of Theorem \ref{thm: rank3KV}] Let $\ell$ be a nef divisor. The case \(\ell=0\) is immediate, so assume \(\ell\neq 0\). By the preceding lemma, write $\ell=a\alpha-\sum_{i=1}^n b_i x_{F_i}.$ We want to show $$\deg\left(\ell(\ell - \alpha +  S_{1,M})\right) \geq 0.$$

We compute
\[
\deg(\ell^2) = a^2 - \sum_{i=1}^n b_i^2, \quad 
\deg(\ell \cdot \alpha) = a, \quad 
\deg(\ell \cdot  S_{1,M}) = \sum_{i=1}^n b_i.
\]
Therefore, $$\deg\left(\ell(\ell - \alpha +  S_{1,M})\right) = a(a-1) -\sum^{n}_{i=1} b_i(b_i-1).$$

Since $\ell$ is nef, we have $\deg(\ell^2) = a^2 - \sum_{i=1}^n b_i^2 \geq 0$. This inequality implies
\[
a(a-1) - \sum_{i=1}^n b_i(b_i-1) \geq 0.
\]
Indeed, by adjoining additional ones to the sequence $(b_i)$ (which do not affect $b_i(b_i-1)$), we may assume $a^2 = \sum_{i=1}^n b_i^2$. In this case, $a \leq \sum_{i=1}^n b_i$, which yields the desired inequality.
\end{proof}
In summary, we have verified Conjecture~\ref{KV vanishing weak} for rank $3$ matroids. For the stronger Conjecture~\ref{KV-vanishing}, however, the inequality becomes
\[
(a-1)(a-2) - \sum_{i=1}^n b_i(b_i-1) \geq 0,
\]
which need not hold without incorporating the combinatorial structure of the matroid. To the author's knowledge, this conjecture remains open even in rank $3$. 

\begin{remark}
The same proof shows that Lemma \ref{lem: big and nef nonzero divisor} generalizes to smooth projective varieties. This result also yields interesting consequences in the matroid setting. Let $M$ be a matroid of rank $r$, and let $\ell_1, \ldots, \ell_k$ be nef divisors with $k < r$. Because $\alpha$ and $\beta$ are big and nef, if $\ell_1 \cdots \ell_k \neq 0$, then $\deg(\alpha^{r-i-k}\beta^{i-1}\ell_1 \cdots \ell_k) > 0$ for any $1 \leq i \leq r - k$. Consequently, if we express the non-zero product $\ell_1 \cdots \ell_k$ as a nonnegative linear combination of $x_{F_1} \cdots x_{F_k}$, there must exist a flag in $\mathcal{H}_{\{i, \ldots, i+k-1\}}$ (cf. Definition~\ref{definition of mathcal H}) with a strictly positive coefficient. This naturally raises the question: what types of flags of flats are guaranteed to appear in products of nef divisors?
\end{remark}

\section{Properties of the \texorpdfstring{$\beta$}{beta} classes} \label{sec:Equations}

While the $\alpha_F$ classes corresponding to flats are fundamental to the intersection theory of matroids, they represent only one part of the geometric picture. In this section, we introduce the $\beta_S$ classes—associated with arbitrary nonempty subsets $S$ of the ground set $E$—and establish their core properties as natural geometric and combinatorial analogues to the $\alpha_F$ classes. 

We structure our investigation into three main areas:
\begin{itemize}
    \item In Section \ref{sec:exceptionalforbeta}, we analyze the behavior of $\beta_S$ under the exceptional isomorphism $\zeta_M$.
    \item In Section \ref{sec:numericaldimensionforbeta}, we compute the numerical dimensions of the $\beta_S$ classes, establishing that intersection products of $\beta$ classes are strictly positive if and only if their corresponding subsets satisfy the dragon--Hall--Rado condition.
    \item In Section \ref{sec:eulercharforbeta}, we investigate the Euler characteristics of $\beta$ classes and prove specific cases of Conjecture \ref{beta conjecture}, concerning the sign of the Euler characteristic for products of positive integral powers of the classes \([\mathcal{L}_S]\) and \([\mathcal{K}_S]\).
\end{itemize}

\subsection{\texorpdfstring{Exceptional isomorphism for $\beta$ classes}{Exceptional isomorphism for beta classes}}\label{sec:exceptionalforbeta}

The classes $[\mathcal{L}_F] \in K(M)$ correspond to divisors $\alpha_F$, and the exceptional isomorphism $\zeta_M$ sends $[\mathcal{L}_F]$ to $1 + \alpha_F + \alpha_F^2 + \cdots$. We first establish the analogous statement for the $\beta_S$ classes.

\begin{definition}
Let $M$ be a matroid on the ground set $E$. For a nonempty subset $S \subseteq E$, let $[\mathcal{K}_S] \in K(M)$ denote the $K$-class of the line bundle with $c_1([\mathcal{K}_S]) = \beta_S$.
\end{definition}

Recall that for a matroid $M$ on a ground set $E$ and an element $i \in E$, there is a deletion map (Lemma \ref{lem:deletion}) $\theta_i: A^\ast(M\setminus i) \rightarrow A^\ast(M)$ given by the projection map between the toric varieties (see \cite[Proposition 3.1]{BHM}). Therefore, there is also a deletion map $\widetilde{\theta}_i: K(M \setminus i) \rightarrow K(M)$ between the $K$-rings, and the map commutes with the Chern class map and $\theta_i$ (see \cite[Theorem 3.2(d)]{Fulton}). That is, the following diagram commutes. \[
\begin{tikzcd}
K(M \setminus i) \arrow[r, rightarrow, "\widetilde{\theta}_i"] \arrow[d, "c_t"'] &
K(M) \arrow[d,  "c_t"] \\
A^*(M \setminus i) \arrow[r, rightarrow, "\theta_i"] & A^*(M)
\end{tikzcd}
\]

\begin{lemma} \label{lem:deletion commutes with zeta}
Let $M$ be a matroid on the ground set $E$, and let $i \in E$. The deletion map $\theta_i$ commutes with the $\zeta$ map. More precisely, we have:
\[
\zeta_M \circ \widetilde{\theta}_i \;=\; \theta_i \circ \zeta_{M \setminus i}.
\]
\end{lemma}

\begin{proof}
Because the $K$-classes of the line bundles $[\mathcal{L}_F]$ generate the $K$-ring (cf.\ Section~\ref{subsec: K Chern}), it suffices to check that $\zeta_M \circ \widetilde{\theta}_i([\mathcal{L}_F])  = \theta_i \circ \zeta_{M \setminus i}([\mathcal{L}_F])$ for any flat $F$ in $M \setminus i$.

For a flat $F$ in $M \setminus i$, $\theta_i(\alpha_F)$ is exactly $\alpha_{\operatorname{cl}(F)}$ in $M$. Since $\widetilde{\theta}_i$ commutes with the Chern class map, we conclude that $\widetilde{\theta}_i([\mathcal{L}_F]) = [\mathcal{L}_{\operatorname{cl}(F)}]$. Therefore, the lemma follows from the fact that $\zeta_{M\setminus i}([\mathcal{L}_F]) = 1 + \alpha_F + \alpha_F^2 + \cdots$.
\end{proof}

In \cite[Theorem 10.11]{BEST}, it is shown that $\zeta_M$ sends $[\mathcal{K}_E]$ to $1 + \beta_E$. We show the analogue for $\beta_S$.

\begin{lemma}\label{lem:composition of deletion of beta}
Let $M$ be a matroid on the ground set $E$ and let $S \subseteq E$ be a nonempty subset. Let $S^c = E \setminus S = \{s_1, \ldots, s_k\}$. Then:
\[
\beta_S = \theta_{S^c}(\beta_{M \setminus S^c}) := \theta_{s_1} \circ \cdots \circ \theta_{s_k} (\beta_{M \setminus S^c}).
\]
Note that $\beta_{S}$ lives in $A^*(M)$, while $\beta_{M \setminus S^c}$ is the $\beta$ class of the matroid $M \setminus S^c$.
\end{lemma}

\begin{proof}
We pick $j \in E \setminus S^c$ and suppose that $\beta_{M \setminus S^c}$ is the sum of flats in $M \setminus S^c$ that do not contain $j$. Through the sequence of $\theta$ maps, every flat that does not contain $j$ and is not contained in $S^c$ will appear, which exactly matches the definition of $\beta_S$.
\end{proof}

\begin{theorem} \label{exceptional morphism of beta}
For nonempty subsets $S\subseteq E$, $\zeta_M$ sends $[\mathcal{K}_S]$ to $1+\beta_S$.
\end{theorem}

\begin{proof}
Let $[\mathcal{K}_{M \setminus S^c}] \in K(M \setminus S^c)$ denote the $K$-class of the line bundle in $M \setminus S^c$ corresponding to its $\beta = \beta_E$ class. Since $\zeta_{M \setminus S^c} ([\mathcal{K}_{M \setminus S^c}]) = 1 + \beta_{M \setminus S^c}$, we conclude that $\zeta_M([\mathcal{K}_S]) = 1 + \beta_S$ by Lemma \ref{lem:deletion commutes with zeta} and Lemma \ref{lem:composition of deletion of beta}.
\end{proof}

\begin{corollary} \label{weak KV for beta}
For nonempty subsets $S\subseteq E$, the divisors $\alpha_S, \beta_S$ lie in $\mathcal{P}$ (cf. Definition \ref{def: P class}). In particular, Conjecture \ref{KV vanishing weak} (the weaker version of Matroid Kawamata--Viehweg Vanishing) holds for any positive (integral) linear combination of $\alpha_S$ and $\beta_S$.
\end{corollary}
\begin{proof}
We have $\zeta_M([\mathcal{L}_S]^{-1}) = 1-\alpha_S$, and $\alpha_S \in \mathcal{P}$. Theorem~\ref{exceptional morphism of beta} shows that $\zeta_M([\mathcal{K}_S]^{-1}) = 1-\beta_S + \beta_S^2 - \cdots$, and $\beta_S \in \mathcal{P}$.
\end{proof}

\begin{corollary}
Let $S_1, \ldots, S_j$ and $T_1, \ldots, T_k$ be subsets of $E$. If $\sum_{i=1}^j \beta_{S_i} = \sum_{i=1}^k \beta_{T_i}$, then $$\prod_{i=1}^j(1+\beta_{S_i}) = \prod_{i=1}^k(1 + \beta_{T_i}).$$
\end{corollary}
\begin{proof}
The equation $\sum_{i=1}^j \beta_{S_i} = \sum_{i=1}^k \beta_{T_i}$ implies $\prod_{i=1}^j [\mathcal{K}_{S_i}] = \prod_{i=1}^k [\mathcal{K}_{T_i}]$. Applying $\zeta_M$ on both sides concludes the proof.
\end{proof}
This identity also holds for the $\alpha_F$'s; however, this is less interesting because $\alpha_F = 0$ if $F$ is a rank 1 flat, and the set $\{\alpha_F: \operatorname{rk}(F)>1\}$ forms a basis of $A^1(M)$.

\begin{corollary}\label{cor:alpha-beta-todd}
Let \(S_1,\dots,S_j\) and \(T_1,\dots,T_k\) be subsets of \(E\), with \(j+k\le r-1\).  Let
\(\operatorname{td}_i\) denote the degree \(i\) component of the Todd class
\(\operatorname{td}(T_M)\).  Then
\[
\deg\!\bigl(\prod_{i=1}^j\alpha_{S_i}\,\prod_{i=1}^k\beta_{T_i}\,\operatorname{td}_{r-j-k-1}\bigr)
=
\deg\!\bigl((-1)^j\prod_{i=1}^j\log(1-\alpha_{S_i})\,
           \prod_{i=1}^k\log(1+\beta_{T_i})\,(1+\alpha+\alpha^2+\cdots)\bigr).
\]
\end{corollary}
\begin{proof}
Compare the coefficient of $s_1 \cdots s_j t_1 \cdots t_k$ in
\[
\chi\!\bigl(\prod_{i=1}^j [\mathcal{L}_{S_i}]^{s_i} \prod_{i=1}^k [\mathcal{K}_{T_i}]^{t_i}\bigr)
\]
as given by Proposition~\ref{prop: HRR} and Equation~\eqref{Euler char exceptional}.   
\end{proof}
\begin{corollary}
For a proper flat $F$ of $M$, let $[\mathcal{X}_F] \in K(M)$ denote the $K$-class of the line bundle with $c_1([\mathcal{X}_F]) = x_F$. We have the following.
\begin{enumerate}
    \item If $F_1$ is a rank 1 flat, then $\zeta_M([\mathcal{X}_{F_1}]^{-1}) = 1 - x_{F_1}$. 
    \item If $F_{r-1}$ is a rank $r-1$ flat, then $\zeta_M([\mathcal{X}_{F_{r-1}}]) = 1 + x_{F_{r-1}}$.
\end{enumerate}
\end{corollary}
\begin{proof}
For (1), one has the identity \(-x_{F_1}=-\beta+\beta_{E\setminus F_1}\).
Using \(\zeta_M([\mathcal{K}_E])=1+\beta\) and \(\zeta_M([\mathcal{K}_{E\setminus F_1}])=1+\beta_{E\setminus F_1}\)
and the relation \(x_{F_1}\beta=0\), a short algebraic computation yields
$\zeta_M([\mathcal{X}_{F_1}]^{-1}) = 1 - x_{F_1}$.  The argument for (2) is analogous, where we consider $\alpha$ and $\alpha_{F_{r-1}}$ instead.
\end{proof}
\begin{remark} \label{facts for excep map} We can compute $\zeta_M([\mathcal{X}_F])$ for general flats $F$ using the same method. Although $\zeta_M([\mathcal{X}_{F_1}]) = 1 + x_{F_1} + x_{F_1}^2 + \cdots + x_{F_1}^{r-1}$ and $\zeta_M([\mathcal{X}_{F_{r-1}}]) = 1 + x_{F_{r-1}}$, it is \emph{false} that $\zeta_M([\mathcal{X}_{F_{r-2}}]) = 1 + x_{F_{r-2}} + x_{F_{r-2}}^2$. 
\end{remark}

\subsection{\texorpdfstring{Numerical dimension of $\beta$ classes}{Numerical dimension of beta classes}} \label{sec:numericaldimensionforbeta}

It is natural to ask whether Theorem~\ref{h vector} extends to products of positive integral powers of \([\mathcal{L}_S]\) and \([\mathcal{K}_S]\). We formulate the following conjecture:

\begin{conjecture}\label{beta conjecture}
Let \(M\) be a matroid and let \([\mathcal{L}]\) be a product of positive integral powers
of the classes \([\mathcal{L}_S]\) and \([\mathcal{K}_S]\).  Let \(d\) be the numerical
dimension of \(c_1([\mathcal{L}])\).  Then
\[
(-1)^d\chi([\mathcal{L}]^{-1})\ge0.
\]
\end{conjecture}

The first question is to compute the numerical dimension of $\beta_S$. 
\begin{definition}
For a matroid $M$ on the ground set $E$ and a sequence $(S_1, \ldots , S_m)$ of nonempty subsets of $E$, we say the sequence satisfies the dragon--Hall--Rado condition (with respect to $M$) if

$$\operatorname{rk}_M\big( \bigcup_{i \in I} S_i \big) \geq 1 + |I|, \text{ for all $\varnothing \subsetneq I \subseteq [m]$}.$$
\end{definition}

For classes $\alpha_F$, it is shown that 
\begin{proposition}[{\cite[Theorem 5.2.4]{BES23}}] 
Let $M$ be a matroid of rank $r$, and let $F_1, \ldots, F_{r-1}$ be flats of $M$ (with repeats allowed). Then $$\deg\big(\alpha_{F_1} \cdots \alpha_{F_{r-1}}\big) = \begin{cases}
1 & \text{if $(F_1, \ldots, F_{r-1})$ satisfies the dragon--Hall--Rado condition}.\\
0 & \text{otherwise}.
\end{cases}$$
This proposition remains true even if we allow $F_i$ to be arbitrary nonempty subsets of $E$ (cf. Definition~\ref{alpha def}). 
\end{proposition}

\begin{theorem} \label{beta nd}
Let $M$ be a matroid of rank $r$, and let $S_1, \ldots, S_{r-1}$ be nonempty subsets of $E$ (with repeats allowed). Then:
\[
\deg\big(\beta_{S_1} \cdots \beta_{S_{r-1}}\big) >  0 \Longleftrightarrow \deg\big(\alpha_{S_1} \cdots \alpha_{S_{r-1}}\big) >  0.
\]
In particular, $\deg\big(\beta_{S_1} \cdots \beta_{S_{r-1}}\big) > 0$ if and only if $(S_1, \ldots, S_{r-1})$ satisfies the dragon--Hall--Rado condition.
\end{theorem}

Before proving this theorem, we establish a few preliminary lemmas. For a flag of flats $\mathcal{F} = (F_1 \subsetneq \cdots \subsetneq F_k)$, we denote the monomial $x_{F_1} \cdots x_{F_k}$ by $x_{\mathcal{F}}$.

\begin{lemma}
Let $M$ be a matroid of rank $r$, and let $\ell_1, \dots, \ell_k$ be nef divisors. The product $\ell_1  \cdots \ell_k$ is a nonnegative linear combination of $x_{\mathcal{F}}$ for flags of flats $\mathcal{F}$.
\end{lemma}

\begin{proof}
We prove the result by induction on $k$. The result is trivially true for $k = 1$.

For a nef divisor $\ell$ and a flag of flats $\mathcal{F} = (F_1 \subsetneq \cdots \subsetneq F_k)$, we can write $\ell = \sum c_Fx_F$ such that $c_{F_i} = 0$ and $c_F \geq 0$ for any $F$ where $F \cup \mathcal{F}$ forms a flag of flats. All other terms will evaluate to zero when multiplied by $x_{\mathcal{F}}$, which completes the induction. 
\end{proof}

If we consider ample divisors instead, the product is a strictly positive linear combination of all possible $x_{\mathcal{F}}$.

\begin{corollary} \label{secretpositive}
Let $D = \sum_F c_Fx_F$ with $c_F \geq 0$, and let $\ell_1, \ldots, \ell_d$ be nef divisors. Then $$D \cdot (\ell_1 \cdots\ell_d) = 0 \Longleftrightarrow c_Fx_F (\ell_1 \cdots\ell_d) = 0 \text{ for all $F$}.$$

\end{corollary}
\begin{proof}
The class $c_Fx_F (\ell_1 \cdots\ell_d)$ can be written as a nonnegative linear combination of $x_{\mathcal{F}}$ for flag of flats $\mathcal{F}$. We may assume $r-d-2 \geq 0$. Let $A$ be an ample divisor. Then
\begin{align*}
\sum c_Fx_F \cdot \ell_1 \cdots \ell_{d} = 0 &\implies \deg(\sum c_Fx_F \cdot \ell_1 \cdots \ell_{d}A^{r-d-2}) = 0 \\ &\implies \deg(c_Fx_F \cdot \ell_1 \cdots \ell_{d}A^{r-d-2})  = 0 \implies c_Fx_F \cdot \ell_1 \cdots \ell_{d} = 0.
\end{align*}

\end{proof}

\begin{proof}[Proof of Theorem \ref{beta nd}]
We prove the following stronger statement: For a set $\varnothing \subsetneq S \subseteq E$ and nef divisors $\ell_1, \ldots, \ell_{d}$, $d < r-1$, we have $$\beta_{S} \cdot \ell_1 \cdots \ell_{d} \neq   0 \Longleftrightarrow \alpha_{S} \cdot \ell_1 \cdots \ell_{d} \neq 0.$$

For any $i \in S$, $\alpha_{S}$ consists of flats containing $i$ but not containing $S$. Therefore, if we sum over all $i \in S$ and let $n = |S|$, we have $$\alpha_{S} = \sum_{j=1}^{n-1} \left(\frac{j}{n} \sum_{|F \cap S| = j} x_F\right).$$

Similarly, for any $i \in S$, $\beta_{S}$ consists of flats not containing $i$ but not contained in $E \setminus S$; thus, we have $$\beta_{S} = \sum_{j=1}^{n-1} \left(\frac{n-j}{n} \sum_{|F \cap S| = j} x_F\right).$$

Therefore, by Corollary \ref{secretpositive},
\begin{align*}
    \beta_{S} \cdot \ell_1 \cdots \ell_{d} \neq  0 &\Longleftrightarrow 
    x_F \cdot \ell_1 \cdots \ell_{d} \neq  0 \text{ for a flat $F$ such that $0 < |F \cap S| < n$} \\&\Longleftrightarrow 
    \alpha_{S} \cdot \ell_1 \cdots \ell_{d} \neq  0. 
\end{align*}

The proof of the original theorem then proceeds by exchanging $\beta_{S_i}$ and $\alpha_{S_i}$ one by one. 
\end{proof}

In particular, the numerical dimension of $\beta_S$ equals $\operatorname{rk}(S) - 1$.

\begin{question}
Do we have a simple formula for $\deg(\beta_{S_1} \cdots \beta_{S_{r-1}})$ if it is nonzero?
\end{question}
\begin{remark}
Let $S = \bigcup^{r-1}_{i=1} S_i$. We have $\deg(\beta_{S_1} \cdots \beta_{S_{r-1}}) \leq \deg(\beta_{S}^{r-1})$, but equality may not hold.

The proof also suggests the notion of $\gamma_S = \alpha_S + \beta_{S}$, which is the sum of $x_F$ where $0 < |F \cap S| < |S|$; this is why we use the notation $\beta_S$ for subtracting flats contained in $E \setminus S$ instead of in $S$.

\end{remark}

\subsection{Euler characteristic for \texorpdfstring{$\beta$}{beta} classes} \label{sec:eulercharforbeta}

Finally, in this section, we investigate the Euler characteristics of $\beta$ classes. For a positive integer $n$, it is known (see \cite[Example 5.8]{EL}) that $$(-1)^{r-1}\chi([\mathcal{K}_E]^{-n}) \geq 0.$$

\begin{corollary}
Let $M$ be a matroid of rank $r$ on the ground set $E$, and let $S \subseteq E$ be a nonempty subset. Setting $S^c = E \setminus S$, suppose that $\operatorname{rk}(S) = i$ (so that $i-1$ is the numerical dimension of $\beta_S$). Let $[\mathcal{K}_{M \setminus S^c}] \in K(M \setminus S^c)$ denote the $K$-class of the line bundle on the matroid $M \setminus S^c$ corresponding to its $\beta = \beta_E$ class. Then for any integer $k$, we have $\chi([\mathcal{K}_S]^k) = \chi([\mathcal{K}_{M \setminus S^c}]^k)$. In particular, for any positive integer $n$, we have
$$
(-1)^{i-1}\chi([\mathcal{K}_S]^{-n}) \ge 0.
$$
\end{corollary}

\begin{proof}
By Lemma~\ref{lem:deletion}, the deletion map commutes with the degree map up to multiplication by a power of $\alpha$, which accounts for any difference in rank between the respective matroids. Furthermore, Lemma~\ref{lem:composition of deletion of beta} states that the composition of deletion maps, denoted $\theta_{S^c}$, sends $\beta_{M \setminus S^c} \in A^*(M \setminus S^c)$ to $\beta_S \in A^*(M)$. One can also check that it maps $\alpha_{M \setminus S^c}$ to $\alpha_M$. 

The Euler characteristic can be evaluated by $
\chi([\mathcal{K}_E]^k) = \deg\left((1+\beta)^k(1+\alpha+ \alpha^2 + \cdots)\right).
$ Applying the composition $\theta_{S^c}$ to this expansion for $\beta_{M \setminus S^c}$ yields $\chi([\mathcal{K}_S]^k) = \chi([\mathcal{K}_{M \setminus S^c}]^k)$.
\end{proof}

To conclude the section, we prove a special case of Conjecture \ref{beta conjecture}.
\begin{lemma}
Let \(F_1,\dots,F_m\) be rank 1 flats such that each
\(S_i:=E\setminus F_i\) has rank \(r\).  Fix positive integers
\(n,a_1,\dots,a_m\) with \(n\ge\sum_{i=1}^m a_i\), and set
\[
B=(n-\sum_{i=1}^m a_i)\beta+\sum_{i=1}^m a_i\beta_{S_i}
 \;=\; n\beta - \sum_{i=1}^m a_i x_{F_i}.
\]
Then \(B\) has numerical dimension \(r-1\). Let $[\mathcal{B}] \in K(M)$ be the $K$-class of the line bundle with $c_1([\mathcal{B}]) = B$. We have
\[
(-1)^{r-1}\chi([\mathcal{B}]^{-1})\;\ge\;0.
\]
\end{lemma}

\begin{proof}[Sketch of proof]
The rank 1 flats \(x_{F_i}\) have vanishing intersections with \(\beta\), so the mixed products of \(\beta\) and the \(x_{F_i}\)'s vanish. Hence, the Chern character of \([\mathcal{B}]^{-1}\) decomposes essentially as a sum of simpler contributions coming from \(-n\beta\) and the \(x_{F_i}\); this leads to the identity
\[
\chi([\mathcal{B}]^{-1})
= \chi([\mathcal{K}_E]^{-n}) - \sum_{i=1}^m \chi([\mathcal{K}_E]^{-a_i}) + \sum_{i=1}^m \chi([\mathcal{K}_{S_i}]^{-a_i}).
\]
(The displayed equality follows from the vanishing of cross terms together
with the relation \(\beta_{S_i}=\beta-x_{F_i}\) for rank-1 flats).

To analyze the first two terms, one uses the deletion–contraction type relation
for \(\chi([\mathcal{K}_E]^{-j})\): for any non-loop, non-coloop element \(i\in E\) (and
assuming the contraction \(M_i\) is loopless) one has
\[
\chi(M,[\mathcal{K}_E]^{-j})=\chi(M\setminus i,[\mathcal{K}_E]^{-j})-\sum_{k=1}^j \chi(M_i,[\mathcal{K}_E]^{-k}).
\]
By induction on the size of the ground set (applying this relation and using
the sign properties for lower-rank contractions) one shows that, for fixed
matroid \(M\) of rank \(r>1\), the sequence
\(\dfrac{(-1)^{r-1}\chi(M,[\mathcal{K}_E]^{-n})}{n}\) is nondecreasing in \(n\).  Consequently
\[
(-1)^{r-1}\chi([\mathcal{K}_E]^{-n}) - \sum_{i=1}^m (-1)^{r-1}\chi([\mathcal{K}_E]^{-a_i}) \;\ge\; 0.
\]

Finally, each term \(\chi([\mathcal{K}_{S_i}]^{-a_i})\) satisfies the expected sign
condition: \((-1)^{r-1}\chi([\mathcal{K}_{S_i}]^{-a_i})\ge0\). Combining the three displayed inequalities yields
\((-1)^{r-1}\chi([\mathcal{B}]^{-1})\ge0\), as required.  (When \(r=1\) the statement is
trivial because the geometry is a point).
\end{proof}

\printbibliography   

@book{Fulton,
  author    = {William Fulton},
  title     = {Intersection Theory},
  edition   = {2nd},
  publisher = {Springer-Verlag},
  year      = {1998},
}

@misc{Staircase,
 author = {Alba, Franquiz Caraballo and Liu, Jeffery},
 title = {A ''{Staircase}'' formula for the {Chern}-{Schwartz}-{MacPherson} cycle of a matroid},
 year = {2024},
 howpublished = {Preprint, {arXiv}:2409.03641 [math.{CO}] (2024)},
 url = {https://arxiv.org/abs/2409.03641},
 arXiv = {arXiv:2409.03641}
}

@misc{Vanishhindex,
 author = {Eur, Christopher and Fink, Alex and Larson, Matt},
 title = {Vanishing theorems for combinatorial geometries},
 year = {2025},
 howpublished = {Preprint, {arXiv}:2510.05207 [math.{AG}] (2025)},
 url = {https://arxiv.org/abs/2510.05207},
 arXiv = {arXiv:2510.05207}
}

@article{AHK18,
 author = {Adiprasito, Karim and Huh, June and Katz, Eric},
 title = {Hodge theory for combinatorial geometries},
 fjournal = {Annals of Mathematics. Second Series},
 journal = {Ann. of Math. (2)},
 issn = {0003-486X},
 volume = {188},
 number = {2},
 pages = {381--452},
 year = {2018},
 doi = {10.4007/annals.2018.188.2.1},
 zbMATH = {6921184},
 Zbl = {1442.14194}
}

@article{Todd,
 author = {Castillo, Federico and Liu, Fu},
 title = {On the {Todd} class of the permutohedral variety},
 fjournal = {Algebraic Combinatorics},
 journal = {Algebr. Comb.},
 issn = {2589-5486},
 volume = {4},
 number = {3},
 pages = {387--407},
 year = {2021},
 doi = {10.5802/alco.157},
 zbMATH = {7367698},
 Zbl = {1467.52022}
}

@article{Nef,
 author = {Gibney, Angela and Maclagan, Diane},
 title = {Lower and upper bounds for nef cones},
 fjournal = {IMRN. International Mathematics Research Notices},
 journal = {Int. Math. Res. Not. IMRN},
 issn = {1073-7928},
 volume = {2012},
 number = {14},
 pages = {3224--3255},
 year = {2012},
 doi = {10.1093/imrn/rnr121},
 zbMATH = {6072501},
 Zbl = {1284.14020}
}

@article{BES23,
 author = {Backman, Spencer and Eur, Christopher and Simpson, Connor},
 title = {Simplicial generation of {Chow} rings of matroids},
 fjournal = {Journal of the European Mathematical Society (JEMS)},
 journal = {J. Eur. Math. Soc. (JEMS)},
 issn = {1435-9855},
 volume = {26},
 number = {11},
 pages = {4491--4535},
 year = {2024},
 doi = {10.4171/JEMS/1350},
 zbMATH = {7927738}
}

@article{Valuative,
 author = {Ferroni, Luis and Schr{\"o}ter, Benjamin},
 title = {Valuative invariants for large classes of matroids},
 fjournal = {Journal of the London Mathematical Society. Second Series},
 journal = {J. Lond. Math. Soc. (2)},
 issn = {0024-6107},
 volume = {110},
 number = {3},
 pages = {86},
 note = {Id/No e12984},
 year = {2024},
 doi = {10.1112/jlms.12984},
 zbMATH = {7926126},
 Zbl = {1548.52017}
}

@article{LLPP2024,
 author = {Larson, Matt and Li, Shiyue and Payne, Sam and Proudfoot, Nicholas},
 title = {{{\(K\)}}-rings of wonderful varieties and matroids},
 fjournal = {Advances in Mathematics},
 journal = {Adv. Math.},
 issn = {0001-8708},
 volume = {441},
 pages = {43},
 note = {Id/No 109554},
 year = {2024},
 doi = {10.1016/j.aim.2024.109554},
 zbMATH = {7823126},
 Zbl = {1535.05266}
}

@article{BEST,
 author = {Berget, Andrew and Eur, Christopher and Spink, Hunter and Tseng, Dennis},
 title = {Tautological classes of matroids},
 fjournal = {Inventiones Mathematicae},
 journal = {Invent. Math.},
 issn = {0020-9910},
 volume = {233},
 number = {2},
 pages = {951--1039},
 year = {2023},
 doi = {10.1007/s00222-023-01194-5},
 zbMATH = {7704061}
}

@article{BHM,
 author = {Braden, Tom and Huh, June and Matherne, Jacob P. and Proudfoot, Nicholas and Wang, Botong},
 title = {A semi-small decomposition of the {Chow} ring of a matroid},
 fjournal = {Advances in Mathematics},
 journal = {Adv. Math.},
 issn = {0001-8708},
 volume = {409},
 pages = {49},
 note = {Id/No 108646},
 year = {2022},
 doi = {10.1016/j.aim.2022.108646},
 zbMATH = {7597104},
 Zbl = {1509.14012}
}

@article{DCP,
 author = {De Concini, C. and Procesi, C.},
 title = {Wonderful models of subspace arrangements},
 fjournal = {Selecta Mathematica. New Series},
 journal = {Selecta Math. (N.S.)},
 issn = {1022-1824},
 volume = {1},
 number = {3},
 pages = {459--494},
 year = {1995},
 doi = {10.1007/BF01589496},
 zbMATH = {838584},
 Zbl = {0842.14038}
}

@article{DF,
 author = {Derksen, Harm and Fink, Alex},
 title = {Valuative invariants for polymatroids},
 fjournal = {Advances in Mathematics},
 journal = {Adv. Math.},
 issn = {0001-8708},
 volume = {225},
 number = {4},
 pages = {1840--1892},
 year = {2010},
 doi = {10.1016/j.aim.2010.04.016},
 zbMATH = {5796750},
 Zbl = {1221.05031}
}

@article{EUR20,
 author = {Eur, Christopher},
 title = {Divisors on matroids and their volumes},
 fjournal = {Journal of Combinatorial Theory. Series A},
 journal = {J. Combin. Theory Ser. A},
 issn = {0097-3165},
 volume = {169},
 pages = {31},
 note = {Id/No 105135},
 year = {2020},
 doi = {10.1016/j.jcta.2019.105135},
 zbMATH = {7137756},
 Zbl = {1428.05149}
}

@misc{Matt_thesis,
  author  = {Matt Larson},
  title   = {K-theoretic positivity for wonderful varieties and matroids},
  year    = {2024},
  howpublished = {\url{https://mattlarson2399.github.io/Papers/LarsonThesis2024.pdf}},
  note    = {Ph.D. thesis},
}

@book{Ox,
  author    = {James Oxley},
  title     = {Matroid Theory},
  edition   = {2nd},
  series    = {Oxford Graduate Texts in Mathematics},
  volume    = {21},
  publisher = {Oxford University Press},
  year      = {2011},
}

@book{Voisin,
  author    = {Claire Voisin},
  title     = {Hodge Theory and Complex Algebraic Geometry I},
  publisher = {Cambridge University Press},
  year      = {2002},
}

@article{JJ,
 author = {Hu, Jiajun and Xiao, Jian},
 title = {Intersection theoretic inequalities via {Lorentzian} polynomials},
 fjournal = {Mathematische Annalen},
 journal = {Math. Ann.},
 issn = {0025-5831},
 volume = {390},
 number = {2},
 pages = {2859--2896},
 year = {2024},
 doi = {10.1007/s00208-024-02822-y},
 zbMATH = {7932370}
}

@article{BH20,
 author = {Br{\"a}nd{\'e}n, Petter and Huh, June},
 title = {Lorentzian polynomials},
 fjournal = {Annals of Mathematics. Second Series},
 journal = {Ann. of Math. (2)},
 issn = {0003-486X},
 volume = {192},
 number = {3},
 pages = {821--891},
 year = {2020},
 doi = {10.4007/annals.2020.192.3.4},
 zbMATH = {7285355},
 Zbl = {1454.52013}
}

@article{KVsurface,
 author = {Xie, Qihong},
 title = {Kawamata-Viehweg vanishing on rational surfaces in positive characteristic},
 fjournal = {Mathematische Zeitschrift},
 journal = {Math. Z.},
 issn = {0025-5874},
 volume = {266},
 number = {3},
 pages = {561--570},
 year = {2010},
 doi = {10.1007/s00209-009-0585-9},
 zbMATH = {5797335},
 Zbl = {1235.14021}
}

@article {FY04,
    AUTHOR = {Feichtner, Eva Maria and Yuzvinsky, Sergey},
     TITLE = {Chow rings of toric varieties defined by atomic lattices},
   JOURNAL = {Invent. Math.},
  FJOURNAL = {Inventiones Mathematicae},
    VOLUME = {155},
      YEAR = {2004},
    NUMBER = {3},
     PAGES = {515--536},
      ISSN = {0020-9910,1432-1297},
   MRCLASS = {14C15 (14M25)},
  MRNUMBER = {2038195},
MRREVIEWER = {G.\ K.\ Sankaran},
       DOI = {10.1007/s00222-003-0327-2},
}

@article{EL,
    AUTHOR = {Eur, Christopher and Larson, Matt},
     TITLE = {K-theoretic positivity for matroids},
   JOURNAL = {Algebr. Geom.},
  FJOURNAL = {Algebraic Geometry},
    VOLUME = {13},
      YEAR = {2026},
    NUMBER = {2},
     PAGES = {241--265},
      ISSN = {2313-1691,2214-2584},
   MRCLASS = {14M25 (52B40)},
  MRNUMBER = {5034709},
}

@article{EHL,
 author = {Eur, Christopher and Huh, June and Larson, Matt},
 title = {Stellahedral geometry of matroids},
 fjournal = {Forum of Mathematics, Pi},
 journal = {Forum Math. Pi},
 issn = {2050-5086},
 volume = {11},
 pages = {48},
 note = {Id/No e24},
 year = {2023},
 doi = {10.1017/fmp.2023.24},
 zbMATH = {7772480},
 Zbl = {1528.05011}
}

@book{positivity,
  author    = {Robert Lazarsfeld},
  title     = {Positivity in Algebraic Geometry I},
  series    = {Ergebnisse der Mathematik und ihrer Grenzgebiete},
  volume    = {48},
  publisher = {Springer-Verlag},
  year      = {2004},
}

@book{Toric,
  author    = {David A. Cox and John B. Little and Hal Schenck},
  title     = {Toric Varieties},
  series    = {Graduate Studies in Mathematics},
  volume    = {124},
  publisher = {American Mathematical Society},
  year      = {2011},
}

\end{document}